\documentclass[12pt,leqno]{article}

\usepackage[lite]{amsrefs}
\usepackage[margin=0.9in, paperwidth=8.5in, paperheight=11in, footskip = 35pt]{geometry}

\usepackage{multirow}

\usepackage{hyperref}
\hypersetup{colorlinks=true}
\usepackage{xcolor}
\hypersetup{
    colorlinks,
    linkcolor={red!60!black},
     citecolor={red!60!black},
     urlcolor={red!60!black}
}

\usepackage{latexsym,enumerate}
\usepackage[notref, notcite]{}

\usepackage{ eucal,tabularx}
\usepackage{amssymb, amsfonts,amsmath}
\usepackage{longtable}
\usepackage{xcolor}
\usepackage{ltablex}

 \usepackage{amsthm}
 \usepackage{thmtools}

\declaretheoremstyle[
  spaceabove=\topsep,
  spacebelow=\topsep,
  headfont=\normalfont\bfseries,
  notefont=\normalfont\bfseries,
  bodyfont=\normalfont\itshape,
  headpunct=.,
]{myplain}

\theoremstyle{myplain}

\newtheorem{theorem}{Theorem}[section]
\newtheorem{lemma}[theorem]{Lemma}

\newtheorem{corollary}[theorem]{Corollary}
\newtheorem{remark}[theorem]{Remark}

\newcommand{\thm}[1]{Theorem~\ref{#1}}
\newcommand{\lem}[1]{Lemma~\ref{#1}}
\newcommand{\cor}[1]{Corollary~\ref{#1}}

\newcommand{\rem}[1]{Remark~\ref{#1}}

\makeatletter
\def\blfootnote{\xdef\@thefnmark{}\@footnotetext}
\makeatother

 \DeclareMathOperator{\dist}{dist}

\newcommand{\sect}[1]{\section{#1}  \setcounter{equation}{0}  }

\newcommand{\norm}[2]{\left\|#1\right\|_{#2}}

\newcommand{\xb}{\bar{x}}

\newcommand{\Y}{\mathbb Y}
\newcommand{\A}{\mathbb A}
\newcommand{\Poly}{\Pi}
\newcommand{\Pn}{\Poly_n}
\newcommand{\N}{\mathbb N}
\newcommand{\Z}{\mathbb Z}
\newcommand{\R}{\mathbb R}
\newcommand{\andd}{\quad\mbox{\rm and}\quad}
\newcommand\w{{\omega}}
\newcommand\e{\varepsilon}

\newcommand{\y}{\upsilon}
\newcommand{\s}{q}

\newcommand{\I}{{\mathcal I}}
 \newcommand{\un}{I_j\cup I_{j+1}}
\newcommand{\ds}{\Delta^{(0)}(Y_{s})}
\newcommand{\tn}{\mathbf{T}_n}

\newcommand{\zn}{\mathbf{Z}_n}
\newcommand{\ozn}{\overline{\mathbf{Z}}_N}

\newcommand{\jj}{{\cal J}}

\newcommand{\tp}{p}
\newcommand{\tr}{{\tilde r}}
\newcommand{\ttr}{{\hat r}}
\newcommand{\tG}{{\widetilde G}}

\newcommand{\tf}{{\widetilde f}}
\newcommand{\tQ}{{\widetilde Q}}

\newcommand{\zbar}{\bar z}
\newcommand{\ibar}{\bar I}
\newcommand{\itilde}{\widetilde I}
\newcommand{\ibari}{\bar I_i}

\newcommand{\intset}{\Upsilon}
\newcommand{\copset}{\Upsilon^{(0)}}

\newcommand{\one}{\widetilde{E}_n(f)}
\newcommand{\intertw}{\widetilde{E}_n(f, Y_s)}
\newcommand{\td}{\widetilde\Delta(f, Y_s)}

 \newcommand{\J}{{\mathcal{J}}}

\def\be  {\begin{equation}}
\def\ee  {\end{equation}}

\def \sgn{\mathop{\rm sgn}\nolimits}
\def \inter{\mathop{\rm int}\nolimits}

\def \pmin{\mathop{\rm pmin}\nolimits}

\newcommand{\ineq}[1]{{\rm(\ref{#1})}}

\newcommand{\ie}{{\em i.e., }}
\newcommand{\eg}{{\em e.g. }}
\newcommand{\st}{\;\; \big| \;\;}

\title{{\sc Onesided, intertwining, positive and copositive polynomial approximation with interpolatory constraints}\thanks{{\it AMS classification:} 41A05, 41A10, 41A15,   41A25, 41A29    {\it Keywords
and phrases:} Intertwining,  copositive, positive, nonnegative, Hermite interpolation,  approximation by algebraic polynomials, exact pointwise estimates}}

\author{
 German Dzyubenko \thanks{Institute of Mathematics  NAS of Ukraine ({\tt dzyuben@gmail.com}). Supported in part by NAS of Ukraine project 0123U100853, and by Johns Hopkins University support to U4U.} \and
Kirill A.  Kopotun\thanks{Department of Mathematics, University of
Manitoba, Winnipeg, Manitoba, R3T 2N2, Canada ({\tt Crimea\_is\_Ukraine@shaw.ca}).  Supported in part by NSERC grant RGPIN-05678-2020.} \and

}

\begin{document}

\maketitle

\begin{abstract}
Given $k\in\N$,  a nonnegative function $f\in C^r[a,b]$, $r\ge 0$, an arbitrary finite collection of points
 $\big\{\alpha_i\big\}_{i\in\J}  \subset  [a,b]$, and a corresponding collection of nonnegative integers
 $\big\{m_i\big\}_{i\in\J}$ with $0\le m_i \le r$, $i\in\J$,
is it true that, for sufficiently large $n\in\N$,
  there exists a polynomial  $P_n$ of degree  $n$   such that
 \begin{enumerate}[(i)]
\item $|f(x)-P_n(x)| \le c \rho_n^r(x) \w_k(f^{(r)}, \rho_n(x); [a,b])$,   $x\in [a,b]$,
where $\rho_n (x):= n^{-1} \sqrt{1-x^2}  +n^{-2}$  and $\w_k$ is the classical $k$-th modulus of smoothness,
\item $P^{(\nu)}(\alpha_i)=f^{(\nu)}(\alpha_i)$, for all $0\le \nu \le m_i$ and all $i\in\J$,
\item[] and
\item   either $P  \ge f$ on $[a,b]$  (\emph{onesided} approximation), or  $P  \ge 0$ on $[a,b]$ (\emph{positive} approximation)?
\end{enumerate}

We provide {\em precise answers}
not only to this question, but also to  similar questions for more general {\em intertwining} and {\em copositive} polynomial approximation.
It turns out that many of these answers are quite unexpected.

We also show that, in general,  similar questions for $q$-monotone approximation with $q\ge 1$ have negative answers, \ie $q$-monotone approximation with general  interpolatory constraints  is impossible if $q\ge 1$.

\end{abstract}

\tableofcontents

\sect{Introduction and Main Results}\label{section1}

In this paper, we are interested in shape preserving approximation    (see \eg \cite{klps2011} for a brief introduction to this area) with an additional restriction that approximating polynomials interpolate the function that is being approximated at a given set of points.

Without any shape restrictions, we can easily do this and also preserve the rate of best approximation in the uniform norm (it is rather evident that this is no longer possible in the integral norms). Indeed, it is sufficient to  correct  the polynomial of best approximation using an appropriate  Lagrange  polynomial obtaining a near best polynomial approximant satisfying the needed interpolation conditions (the constant will necessarily depend on the minimum distance among all interpolation points).

Hence,  the following question is rather natural:
 If $f$ has a certain shape that we wish to preserve while approximating it by polynomials,  is interpolation at a given set of points possible and, if so, what can be said about errors of approximation of $f$ by such polynomials?
Our goal is to provide answers to this question.

\medskip

In order to discuss this further, we now recall some standard notations and definitions (see \eg \cite{klps2011}).
We say that a function $f$ is   $q$-monotone %(sometimes, it is called `$q$-convex')
on an interval $J$ and write $f\in \Delta^{(q)}(J)$) if, for all collections of $q+1$ distinct points $t_0, t_1, \dots, t_q$ in $J$, the divided differences $[t_0, \dots, t_q; f]$ are nonnegative. In particular, $\Delta^{(0)}(J)$, $\Delta^{(1)}(J)$ and $\Delta^{(2)}(J)$ are the sets of all nonnegative, nondecreasing and convex functions on $J$, respectively. If $q\ge 3$, then functions $f$ from $\Delta^{(q)}(J)$ are $(q-2)$-times differentiable in  $\inter J$ (the interior of $J$) and $f^{(q-2)} \in \Delta^{(2)}(\inter J)$.

Let $\Y_s$, $s\in\N$, be the set of all collections $Y_s:=\big\{y_i\big\}_{i=1}^s$ of points $y_i$, such that $-1<y_s<\dots<y_1<1$, and define $y_0 := 1$ and $y_{s+1} := -1$.
We denote by $\Delta^{(q)}(Y_s)$ the set of functions which change their $q$-monotonicity at the points in $Y_s$. More precisely,
\[
f \in \Delta^{(q)}(Y_s) \quad \iff \quad (-1)^i f \in \Delta^{(q)}([y_{i+1}, y_i]), \; 0\le i \le s .
\]
Also, if $s=0$, then $\Y_0 := \{\emptyset\}$ and $\Delta^{(q)}(Y_0) := \Delta^{(q)}(I)$, where $I := [-1,1]$.

Let $C^r := C^r(I)$ denote  the space of all $r$-times continuously differentiable functions on $I$, and
let $\A_p$, $p\in\N$, be the set of all collections
$A_p:=\big\{\alpha_i\big\}_{i=1}^p$    of points $\alpha_i$ such that $-1 \le  \alpha_p< \dots <\alpha_1 \le 1$. For a given $f \in C^r(I)$, $r\in\N_0$,  we denote
%by $\I^{(r)}(f, A_p)$ the set of all functions $g:  I\mapsto\R$ such that $g^{(i)}(\alpha)=f^{(i)}(\alpha)$, for all $0\le i \le r$ and $\alpha\in A_p$:
\[
\I^{(r)}(f, A_p) := \left\{ g\in C^r(I)  \st g^{(i)}(\alpha)=f^{(i)}(\alpha), \; \text{for all $\alpha\in A_p$ and $0\le i \le r$} \right\} .
\]
 In particular, if $r=0$, we use the notation $\I (f, A_p) := \I^{(0)}(f, A_p)$.  For convenience, we also denote   $\A_0 := \{\emptyset\}$.
Obviously,  if $f\in C^r(I)$, $r\ge 1$, then
$\I^{(m+1)}(f, A_p) \subset \I^{(m)}(f, A_p)$, for all $0\le m \le r-1$.

As usual,
$\Pn$ is the set of all polynomials of degree $\le n$, $\norm{\cdot}{J}$ is the uniform norm on $J$,  $\w_k(f, \cdot; J)$ is the $k$-th modulus of smoothness of $f$ on $J$,
%$\w_k(f; J) := \w_k(f, |J|; J)$,
% $\varphi(x) := \sqrt{1-x^2}$,
$\rho_n (x):= n^{-1} \sqrt{1-x^2}  +n^{-2}$ and $\N_0 := \N\cup \{0\}$.
Throughout this paper,  $c(\dots)$  denote positive constants which depend only on the parameters in the parentheses. Also, the interval $I=[-1,1]$ will be omitted from the notation, \ie
 $\norm{\cdot}{} := \norm{\cdot}{I}$, $\w_k(f, \cdot) := \w_k(f, \cdot; I)$, etc.

Throughout this paper, if sets $Y_s$ and $A_p$   both appear in a statement, it is  assumed that they have no points in common, \ie  {\em  the restriction $Y_s \cap A_p = \emptyset$ is always assumed but is not explicitly stated.}
 Note that this has no influence on the generality of the obtained results and
 is done    for convenience only.

\subsection{If $q\ge 1$, then  $q$-monotone approximation with general  interpolatory constraints  is impossible}

We start with the following observation that implies that $q$-monotone polynomial approximation with interpolatory constraints  is, in general, not possible if $q\ge 1$.

\begin{lemma} \label{lem1}
For any  $q\in\N$, $s\in\N_0$ and $Y_s\in\Y_s$, there exists $f\in  \Delta^{(q)}(Y_s)$ and a set $A_p\in\A_p$, $p\ge q+2$, such that, for any $n\in\N_0$,
\[
\Delta^{(q)}(Y_s) \cap \I(f, A_p) \cap \Pn = \emptyset .
\]
\end{lemma}

 \begin{proof}
Given  $q\in\N$, $s\in\N_0$ and $Y_s\in\Y_s$, we set $\xi := (1+y_1)/2$ and  pick  $A_p$, $p\ge q+2$, so that  $y_1< \alpha_{q+2} < \dots < \alpha_2 < \xi < \alpha_1$. Then,
 $f := (\cdot -\xi)_+^q$   is clearly in $\Delta^{(q)}(Y_s)$.
 We now recall the following property of $q$-monotone functions (see \cite{b}*{Theorem 5}):
 \begin{quote}
 Let $f\in \Delta^{(q)}(J)$, $q\in\N$, and let $l_{q-1} := l_{q-1}(z_q, \dots, z_1)   \in\Poly_{q-1}$ be the Lagrange polynomial interpolating $f$ at $q$ distinct points in $J$:
 $z_q < z_{q-1} < \dots < z_1$.
  Then, $(-1)^i \left(  f(x) -l_{q-1}(x) \right) \ge 0$, for $x\in [z_{i+1}, z_i]$, $1\le i \le q$, and $f(x) -l_{q-1}(x)  \ge 0$, for $x \ge z_1$. In other words,
  \be \label{bullen}
  f -l_{q-1} \in \Delta^{(0)} \left(\{z_i\}_{i=1}^{q} \right) .
  \ee
 \end{quote}
 Note that this result is a generalization of the property that a convex function always lies below its chord.

 Suppose now that $p_n \in \Delta^{(q)}(Y_s) \cap \I(f, A_p) \cap \Pn$. Then, $p_n$ is $q$-monotone on $[y_1,1]$, $p_n(\alpha_i)=0$, $2\le i \le q+2$ and $p_n(\alpha_1) = f (\alpha_1) >0$.
 Now, polynomials $\ell_1 := l_{q-1}(\alpha_{q+2}, \dots, \alpha_3) \equiv 0$ and
 $\ell_2:= l_{q-1}(\alpha_{q+2}, \dots, \alpha_4, \alpha_2) \equiv 0$ (in the case $q=1$, $\ell_2:= l_{0}(\alpha_2) \equiv 0$)
 interpolate $p_n$ at $q$ points in $[y_1,1]$, and so it follows from \ineq{bullen} that $p_n(x) = p_n(x) - \ell_1(x) \ge 0$, for $x \ge \alpha_3$, and
 $-p_n(x) = (-1) \left( p_n(x) - \ell_2(x) \right) \ge 0$, for $\alpha_4 \le x \le \alpha_2$. Therefore, $p_n(x) =0$, $\alpha_3 \le x \le \alpha_2$, and so $p_n \equiv 0$. Hence, $p_n(\alpha_1)  >0$ does not hold.
 This contradiction implies that such a polynomial $p_n$ does not exist.
\end{proof}

We remark that, if the number of points $p$ in $A_p \in \A_p$ is sufficiently small, then $q$-monotone approximation with interpolatory constraints may be possible, but it may or may not be as good as $q$-monotone approximation without  interpolation at the points in $A_p$. For example, if $q\ge 1$, then
 denoting
\[
 E_n^{(q)}(f, Y_s; A_p) := \inf \left\{ \norm{f-p_n}{} \st    p_n\in\Delta^{(q)}(Y_s) \cap \I (f, A_p) \cap \Pn \right\} ,
\]
we have
\[
E_n^{(q)}(f, Y_s; A_1) \le 2 E_n^{(q)}(f, Y_s; \emptyset), \quad \text{for any $A_1 = \{\alpha\} \subset I$,}
\]
and
\[
E_n^{(q)}(f, Y_s; A_{q+1}) = 1, \quad \text{if $Y_s \subset [-1,0)$, $f(x) = x_+^q \in \Delta^{(q)}(Y_s)$ and $A_{q+1} \subset (y_1,0)$.}
\]

\subsection{Positive, copositive, onesided and intertwining approximation with interpolatory constraints: definitions and a motivational example} \label{sectionpc}

Note that the case for $q=0$ (copositive approximation) is not covered by \lem{lem1}. In fact, as shown below, (co)positive polynomial approximation requiring interpolation on general sets $A_p$ is always
 possible. However,  its rate of approximation is often worse than the rate of (co)positive approximation  without these additional interpolatory constraints.

\medskip

Another type of shape preserving approximation that is related to (co)positive approximation is the so-called onesided and, more generally, {\em intertwining} approximation.
We say (see \cite{hky}; as far as we know, this is where intertwining approximation was first introduced and discussed) that $\{P,Q\}$ is an {\em intertwining pair} of functions (polynomials, splines, etc.) {\em for $f$ with respect to
$Y_s\in\Y_s$}   if
\[
P-f, f-Q\in\ds .
\]
 In particular, in the case $s=0$, we have $Q\le f \le P$ on $I$, and approximation of $f$ by $P$ and $Q$ is usually referred to as {\em onesided} approximation.

  %For convenience, in the case $s=0$, we still refer to $\{P,Q\}$ as an {\em intertwining pair for $f$.}

The error of intertwining/onesided polynomial approximation is usually defined as the infimum of $\norm{P-Q}{}$ taken over all intertwining pairs of polynomials for $f$. In this paper, we also work with an equivalent quantity (and equivalent corresponding pointwise estimates)
\[
\intertw := \inf \left\{\|f-P\|  \st   P \in\Pn \cap \td  \right\}  ,
\]
where
\[
\td := \left\{ P: I\mapsto\R  \st P-f \in \ds \right\}.
\]
The error of onesided approximation is denoted by  $\one := \widetilde{E}_n(f, Y_0)$.

Estimates for intertwining/onesided approximation immediately imply those for copositive/positive approximation since, if $f \in \Delta^{(0)}(Y_s)$ and $P\in \td$, then
$P = (P-f)+f \in \Delta^{(0)}(Y_s)$.
However, as we show below, there are many cases when copositive (polynomial) approximation (with or without interpolatory constraints) is possible while intertwining   approximation is not.

 \medskip

It is well known that the `correct' pointwise estimates for polynomial approximation need to involve the quantity $\rho_n(x)$ since, without it, matching inverse results would not be possible.
Namely, the classical Timan-Dzyadyk-Freud-Brudnyi direct theorem for the approximation by algebraic polynomials can be stated as follows (see \eg  \cites{kls21} for discussions):
\begin{quote}
if $k\in\N$, $r\in\N_0$ and $f\in C^r$, then for each $n\geq k+r-1$ there is a polynomial $p_n\in\Pn$ satisfying
\be \label{classdir}
|f(x)-p_n(x)| \le c(k,r) \rho_n^r(x) \w_k(f^{(r)}, \rho_n(x)) , \quad x\in I .
\ee
\end{quote}

Hence,  estimates of type \ineq{classdir}  is what we are interested in while discussing approximation with interpolatory constraints (see also Section~\ref{point} for discussions of how interpolation  yields further improvement of
these  estimates).

\medskip \noindent
{\bf Illustrative  example:}
 Suppose  that $r\in\N_0$  is fixed, and  two functions $f,g\in C^r$ are   such that $g\in\Delta^{(0)}(I)$ and $f\in \Delta^{(0)}(\{0\})$
(\ie  $g\ge 0$ on $[-1,1]$, and $f\le 0$ on $[-1,0]$ and $f\ge 0$ on $[0,1]$), and recall that, for $0\le m\le r$,     $\I^{(m)}(f, \{\alpha\})$  denotes the set of   functions $h\in C^m$ such that
$h^{(i)}(\alpha)=f^{(i)}(\alpha)$ for all   $0\le i \le m$.
%Obviously, $\I^{(m+1)}(f, \{\alpha\}) \subset \I^{(m)}(f, \{\alpha\})$.

\medskip

We are interested in {\em precise} answers to the following questions:
\begin{enumerate}[\bf Q1.]
\item  For which $m$ does there exist $p_n\in\Pn\cap \I^{(m)}(g, \{1/2\})$  such that $p_n \ge g$ on $I$ and \ineq{classdir} (with $f$ replaced by $g$)  holds? (\emph{onesided approximation})
\item For which $m$ does there exist $p_n\in\Pn\cap \I^{(m)}(g, \{1/2\})$  such that $p_n \ge 0$ on $I$ and \ineq{classdir} (with $f$ replaced by $g$) holds? (\emph{positive approximation})
\item For which $m_1$ and $m_2$  does there exist $p_n\in\Pn\cap \I^{(m_1)}(f, \{1/2\}) \cap \I^{(m_2)}(f, \{0\})$  such that $p_n \le f$ on $[-1,0]$,  $p_n \ge f$ on $[0,1]$,  and \ineq{classdir} holds? (\emph{intertwining approximation})
\item For which $m_1$ and $m_2$  does there exist $p_n\in\Pn\cap \I^{(m_1)}(f, \{1/2\}) \cap \I^{(m_2)}(f, \{0\})$  such that $p_n \in \Delta^{(0)}(Y_s)$    and \ineq{classdir} holds? (\emph{copositive approximation})
\end{enumerate}

The answers are rather unexpected:
\begin{enumerate}[\bf {A}1.]
\item Such $p_n$ does not exist if $r=0$ or $r=1$. If $r\ge 2$ is even, then $p_n$ exists if $m=r-1$ and does not exist if  $m=r$. If $r\ge 3$ is odd, then $p_n$ exists if $m=r-2$ and does not exist if  $m=r-1$.
\item If $r=0$ and $m=0$, then such $p_n$ exists  but only if $k=2$ ($k$ is the order of the modulus of smoothness in \ineq{classdir}). If $k=3$, then such $p_n$ no longer exists.
If $r=1$ and $m=0$, then  such $p_n$ exists  if $k=3$ and does not exist  if $k=4$. If $r=1$ and $m=1$, then such $p_n$ does not exist (for any $k$).
If $r\ge 2$, then the answers are the same as in ``A1''.
\item Such $p_n$ does not exist if $r=0$ or $r=1$. If $r\ge 2$ is even, then such $p_n$ exists if $m_1=r-1$ and $m_2=r-2$, and does not exist if either $m_1=r$ or $m_2=r-1$.
If $r\ge 3$ is odd, then such $p_n$ exists if $m_1=r-2$ and $m_2=r-1$, and does not exist if either $m_1=r-1$ or $m_2=r$.
\item This is the most involved  answer.
\begin{itemize}
\item[$\mathbf{r=0}$:] If $m_1=m_2=0$, then such $p_n$ exists   if $k=2$ and  does not exist  if $k=3$.
\item[$\mathbf{r=1}$:] If $m_1=m_2=0$, then such $p_n$ exists  if $k=3$ and does not exist  if $k=4$. If   $m_1=1$, then such $p_n$ does not exist.
If   $m_1=0$ and $m_2=1$, then such $p_n$ exists   if $k=2$ and does not exist  if $k=3$.
\item[$\mathbf{r=2}$:] If   $m_1=1$ and $m_2=0$, then such $p_n$ exists (for any $k$). However, $m_1=1$ cannot be increased  to $2$. If   $m_1=1$ and $m_2=1$, then such $p_n$ exists  if $k=3$ and does not exist  if $k=4$.
If   $m_2=2$, then such $p_n$ does not exist.
\item[$\mathbf{r\ge 3}$:] The answers are the same as in ``A3''.
\end{itemize}
\end{enumerate}

\subsection{Summary of main results}

We start this section with the summary of
 all known results involving pointwise estimates  for (co)positive,  onesided and intertwining polynomial approximation {\em without} any additional interpolatory restrictions.

\begin{longtable}{|c|c|c|}
 \caption{\mbox{Approximation without  interpolatory constraints}} \label{tablewithout}  \\ \hline
\multicolumn{3}{|c|}{\bf Onesided approximation}\\ \hline
$f\in C$ & $\exists$   $P_n \geq f$:
    $|f(x)-P_n(x)| \leq c \w_k(f,\rho_n(x))$ & \cite{hky}*{Theorem 1} \\
\hline \hline
\multicolumn{3}{|c|}{\bf Positive approximation}\\ \hline
\multirow{2}{*}{$f\in C$} & \multirow{2}{*}{$\exists$  $P_n \geq 0$:   $|f(x)-P_n(x)| \leq c \w_k(f,\rho_n(x))$} & \cite{dz96}*{Theorem 3},  \\
& & or \cite{hky}*{Corollary 2}
\\ \hline \hline
\multicolumn{3}{|c|}{\bf Intertwining approximation}\\ \hline
$f\in C$ & $\widetilde{E}_n(f, Y_s)  \not \leq c \|f\|  $ &  \cite{hky}*{Theorem 13} \\ \hline
$f\in C^1$ &  $\exists$  $ P_n\in\td$:
 $|f(x)-P_n(x)| \leq c \rho_n(x) \w_k(f',\rho_n(x))$ & \cite{hky}*{Theorem 5 } \\ \hline \hline
\multicolumn{3}{|c|}{\bf Copositive approximation}\\ \hline
\multirow{2}{*}{$f\in C$} & $\exists$ $P_n$, copositive with $f$:
  $|f(x)-P_n(x)| \leq c \w_3(f,\rho_n(x))$ & \cite{hky}*{Theorem 7} \\
\cline{2-3}
& $\w_3(f,\rho_n(x))$ cannot be replaced by $\w_4(f,n^{-1})$
                      & Zhou \cites{zhou,zhou-atappl} \\ \hline
\multirow{2}{*}{$f\in C^1$} & $\exists$ $P_n$, copositive with $f$: &  \cite{dz96}*{Theorem 1}, \\
                            &   $|f(x)-P_n(x)| \leq c \rho_n(x) \w_k(f',\rho_n(x))$ & or  \cite{hky}*{Corollary 6} \\
   \hline
\end{longtable}

The main goal of this paper is to investigate what changes in Table~\ref{tablewithout} are needed if we add the requirement that approximating polynomials also interpolate the function and possibly  its derivatives (\ie Hermite interpolation)  at the points in $A_p \in\A_p$.
Additionally, it is clear that any continuous function from $\td$ automatically interpolates $f$ at all points in $Y_s\in\Y_s$ (\ie $\td \subset  \I (f, Y_s)$), and
we also investigate if interpolation of derivatives of $f$  at the points in $Y_s$ is possible as well as its limitations.

It turns out that whether optimal estimates are possible depends not only on  the smoothness of $f$, but also on whether the order of the highest derivative of $f$ is even or odd, and on the location of the interpolation points (\ie on whether $A_p$ contains any points in the  interior of $I$). Hermite interpolation at the points from $Y_s$ is even more involved.

For clarity, we summarize our main results  in the following four tables.

\newcommand{\onesize}{0.55in}
\newcommand{\twosize}{3.5in}
\newcommand{\threesize}{1.8in}

\begin{longtable}{|c|c|c|}
 \caption{Onesided approximation} \label{onetable}  \\ \hline
\multicolumn{3}{|c|}{\bf Onesided  approximation with interpolatory constraints}\\ \hline
$f\in C  $ & not possible in general  &   \lem{osneg} with $r=0$    \\ \hline
 \multirow{4}{*}{$f\in C^1$}  & not possible if $A_p \cap \inter I \ne \emptyset$   &   \lem{osneg}(i) with $r=1$      \\ \cline{2-3}
                              &   $\exists$    $P_n \geq f$ and   $ P_n \in \I (f, \{\pm 1\})$:  &   \thm{thone} with \\
                              & $|f(x)-P_n(x)| \leq c \rho_n(x) \w_k(f',\rho_n(x)) $  &      $(r,m_1,m_3)=(1,-1,0)$   \\   \cline{2-3}
                              &  $\I (f, \{\pm 1\})$  cannot be replaced by  $\I^{(1)} (f, \{\pm 1\})$ & \lem{osneg}(ii) ($r=1$)   \\
  \hline
 \multirow{4}{*}{\parbox[t]{\onesize}{\centering $f\in C^r$, $r\ge 2$ is even}}  &
  \multirow{3}{*}{\parbox[t]{\twosize}{\centering
 $\exists$    $P_n \geq f$ and    $P_n \in \I^{(r-1)}(f, A_p)$:
 $|f(x)-P_n(x)| \leq c \rho_n^r(x) \w_k(f^{(r)},\rho_n(x)) $}} & \thm{thone} with \\
&                                                              &       $(r,m_1,m_3)$      \\
&                                                              &       $=(2j,2j-1,2j-1)$      \\ \cline{2-3}
&  $\I^{(r-1)}(f, A_p)$  cannot be replaced by  $\I^{(r)}(f, A_p)$ & \lem{osneg}
    \\
\hline
 \multirow{8}{*}{\parbox[t]{\onesize}{\centering $f\in C^r$, $r\ge 3$ is odd}}
     &  $\exists$    $P_n \geq f$ and     $ P_n \in  \I^{(r-2)}(f, A_p)$:  &   \multirow{2}{*}{\parbox[t]{\threesize}{\centering follows from a stronger statement below}}    \\
     & $|f(x)-P_n(x)| \leq c \rho_n^r(x) \w_k(f^{(r)},\rho_n(x)) $  &          \\ \cline{2-3}
     &  \parbox[t]{\twosize}{\centering $\I^{(r-2)}(f, A_p)$  cannot be replaced by  $\I^{(r-1)}(f, A_p)$ unless $A_p=\{\pm 1\}$} & \multirow{2}{*}{\parbox[t]{\threesize}{\centering \lem{osneg}(i)}} \\ \cline{2-3}
     &   \multirow{3}{*}{\parbox[t]{\twosize}{\centering
              $\exists$    $P_n \geq f$ and   $ P_n \in  \I^{(r-2)}(f, A_p)\cap \I^{(r-1)}(f, \{\pm 1\})$:
              $|f(x)-P_n(x)| \leq c \rho_n^r(x) \w_k(f^{(r)},\rho_n(x)) $ }} & \thm{thone} with \\
     &                                                                        &   $(r,m_1,m_3)$    \\
     &                                                                        &   $=(2j+1,2j-1,2j)$    \\ \cline{2-3}
     &  $\I^{(r-1)}(f, \{\pm 1\})$  cannot be replaced by  $\I^{(r)}(f, \{\pm 1\})$ & \lem{osneg}(ii)  \\
     \hline
\end{longtable}

\begin{longtable}{|c|c|c|}
 \caption{Positive approximation} \label{postable}  \\ \hline
 \multicolumn{3}{|c|}{\bf Positive  approximation with interpolatory constraints ($f\ge 0$)}\\ \hline
% %
%   \blfootnote{Table continued on next page}
%   %
 & $\exists$ $P_n\ge 0$ and $P_n\in \I(f, A_p)$: &    \thm{thpositive} with   \\
$f\in C$ & $|f(x)-P_n(x)| \leq c \w_2(f,\rho_n(x))$ &   $(r,k,m_1,m_3) = (0,2,0,0)$        \\
 \cline{2-3}
   &   $\w_2(f,\rho_n(x))$ cannot be replaced by  $\w_3(f,\rho_n(x))$  &   \lem{lemma23}    \\ \hline
 \multirow{8}{*}{$f\in C^1$} & $\exists$ $P_n\ge 0$ and $P_n\in \I(f, A_p)$: &  \thm{thpositive} with\\
                             & $|f(x)-P_n(x)| \leq c  \rho_n(x)  \w_3(f',\rho_n(x))$  &  $(r,k,m_1,m_3) = (1,3,0,0)$   \\  \cline{2-3}
                             & \parbox[t]{\twosize}{\centering  $\w_3(f',\rho_n(x))$ cannot be replaced by  $\w_4(f',\rho_n(x))$   unless $A_p = \{\pm 1\}$}   & \multirow{2}{*}{\lem{lemma205}} \\  \cline{2-3}
                             &   $\I (f, A_p)$ cannot   be replaced by  $\I^{(1)}(f, A_p)$   &  \lem{lemma28}  \\ \cline{2-3}
                             & $\exists$ $P_n\ge 0$ and $P_n\in \I(f, \{\pm 1\})$: &   \thm{thpositive} with  \\
                             & $|f(x)-P_n(x)| \leq c  \rho_n(x)  \w_k(f',\rho_n(x))$  &    $(r,k,m_1,m_3) = (1,k,-1,0)$       \\  \cline{2-3}
                             &   $\I(f, \{\pm 1\})$  cannot   be replaced by  $\I^{(1)}(f, \{\pm 1\})$   &  \lem{lemma28}    \\ \hline
 \multirow{4}{*}{\parbox[t]{\onesize}{\centering $f\in C^r$, $r\ge 2$ is even}}
                 &  \multirow{3}{*}{\parbox[t]{\twosize}{\centering     $\exists$ $P_n\ge 0$ and $P_n\in \I^{(r-1)}(f, A_p)$:
                              $|f(x)-P_n(x)| \leq c \rho_n^r(x) \w_k(f^{(r)},\rho_n(x)) $}} & \thm{thpositive} with \\
                 &                                                                          &   $(r,k,m_1,m_3)$     \\
                 &                                                                          &   $= (2j,k,2j-1,2j-1)$     \\  \cline{2-3}
                 &   $\I^{(r-1)}(f, A_p)$ cannot   be replaced by  $\I^{(r)}(f, A_p)$   &  \lem{lemma30}  \\ \hline
 \multirow{8}{*}{\parbox[t]{\onesize}{\centering $f\in C^r$, $r\ge 3$ is odd}}
                 &  $\exists$ $P_n\ge 0$ and $P_n\in \I^{(r-2)}(f, A_p)$:  &  \multirow{2}{*}{\parbox[t]{\threesize}{\centering follows from a stronger statement below}}    \\
                 & $|f(x)-P_n(x)| \leq c \rho_n^r(x) \w_k(f^{(r)},\rho_n(x)) $  &                        \\ \cline{2-3}
                 &  \parbox[t]{\twosize}{\centering $\I^{(r-2)}(f, A_p)$  cannot be replaced by  $\I^{(r-1)}(f, A_p)$ unless $A_p=\{\pm 1\}$}  &  \multirow{2}{*}{\lem{lemma50}}   \\ \cline{2-3}
                 &   \multirow{3}{*}{\parbox[t]{\twosize}{\centering  $\exists$  $P_n\ge 0$ and $ P_n \in \I^{(r-2)}(f, A_p)\cap \I^{(r-1)}(f, \{\pm 1\})$:
                                             $|f(x)-P_n(x)| \leq c \rho_n^r(x) \w_k(f^{(r)},\rho_n(x)) $}} & \thm{thpositive} with \\
                 &                                                                                         & $(r,k,m_1,m_3)$ \\
                 &                                                                                         &  $= (2j+1,k,2j-1,2j)$  \\    \cline{2-3}
                 & $ \I^{(r-1)}(f, \{\pm 1\})$ cannot   be replaced by $ \I^{(r)}(f, \{\pm 1\})$ &  \lem{lemma28} \\  \hline
\end{longtable}

%\begin{longtable}{|c|p{7.9cm}|c|}
\begin{longtable}{|c|c|c|}
 \caption{Intertwining approximation} \label{inttable}  \\ \hline
%\footnotetext{\multicolumn{2}{r}{\footnotesize( To be continued)}}
%\renewcommand{\thefootnote}{\fnsymbol{footnote}}
%\renewcommand{\footnoterule}{}
%\footnotetext[0]{First footnote on the caption}
%\endfoot
%\endlastfoot
 %
 \multicolumn{3}{|c|}{\bf Intertwining approximation with interpolatory constraints}\\ \hline
 %
 %
% \footnotemark[2]
 %\footnotetext[0]{First footnote on the caption}
% \blfootnote{Table continued on next page}
 %
 %
$f\in C  $ & not possible in general  &   \lem{osneg} ($r=0$)    \\ \hline
 \multirow{6}{*}{$f\in C^1$}  & not possible if $A_p \cap \inter I \ne \emptyset$   &   \lem{osneg}(i) ($r=1$)     \\ \cline{2-3}
                              &   \multirow{3}{*}{\parbox[t]{\twosize}{\centering    $\exists$   $ P_n \in\td\cap  \I (f, \{\pm 1\})$:
                              $|f(x)-P_n(x)| \leq c \rho_n(x) \w_k(f',\rho_n(x)) $}} & \thm{th:inti} with  \\
    &                                                                                &  $(r,m_1,m_2, m_3)$ \\
    &                                                                                &  $= (1,-1,0,0)$ \\   \cline{2-3}
                              &  $\I (f, \{\pm 1\})$  cannot be replaced by  $\I^{(1)} (f, \{\pm 1\})$ & \lem{osneg}(ii) ($r=1$) \\ \cline{2-3}
                              &  \parbox[t]{\twosize}{\centering $\td$  cannot be replaced by  $\td\cap \I^{(1)}(f, Y_s)$} & \multirow{2}{*}{\lem{interneg} ($r=1$)}  \\  \hline
 \multirow{5}{*}{\parbox[t]{\onesize}{\centering $f\in C^r$, $r\ge 2$ is even}}  &
 \multirow{3}{*}{\parbox[t]{\twosize}{\centering  $\exists$   $P_n \in\td\cap \I^{(r-1)}(f, A_p) \cap \I^{(r-2)}(f, Y_s)$:
 $|f(x)-P_n(x)| \leq c \rho_n^r(x) \w_k(f^{(r)},\rho_n(x)) $}} & \thm{th:inti} with \\
 &                                                              & $(r,m_1,m_2, m_3)$ \\
 &                                                              & $=(2j,2j-1,2j-2,2j-1)$ \\ \cline{2-3}
%\uphand
&  $\I^{(r-1)}(f, A_p)$  cannot be replaced by  $\I^{(r)}(f, A_p)$ & \lem{osneg}   \\ \cline{2-3}
&  $\I^{(r-2)}(f, Y_s)$  cannot be replaced by  $\I^{(r-1)}(f, Y_s)$ & \lem{interneg}   \\
\hline
 \multirow{9}{*}{\parbox[t]{\onesize}{\centering $f\in C^r$, $r\ge 3$ is odd}}
     &  $\exists$   $ P_n \in\td\cap \I^{(r-2)}(f, A_p)\cap \I^{(r-1)}(f, Y_s)$:  &   \multirow{2}{*}{\parbox[t]{\threesize}{\centering follows from a stronger statement below}} \\
     & $|f(x)-P_n(x)| \leq c \rho_n^r(x) \w_k(f^{(r)},\rho_n(x)) $  &            \\ \cline{2-3}
     &  \parbox[t]{\twosize}{\centering $\I^{(r-2)}(f, A_p)$  cannot be replaced by  $\I^{(r-1)}(f, A_p)$ unless $A_p=\{\pm 1\}$} & \multirow{2}{*}{\parbox[t]{\threesize}{\centering \lem{osneg}(i)}}  \\ \cline{2-3}
     &  \parbox[t]{\twosize}{\centering $\I^{(r-1)}(f, Y_s)$  cannot be replaced by  $\I^{(r)}(f, Y_s)$} & \lem{interneg}  \\ \cline{2-3}
     &
      \multirow{3}{*}{\parbox[t]{\twosize}{\centering    $\exists$   $ P_n \in\td\cap  \I^{(r-2)}(f, A_p)\cap \I^{(r-1)}(f, Y_s)  \cap  \I^{(r-1)}(f, \{\pm 1\})$:
      $|f(x)-P_n(x)| \leq c \rho_n^r(x) \w_k(f^{(r)},\rho_n(x)) $}}   &      \thm{th:inti} with      \\
&                                                                     & $(r,m_1,m_2, m_3)$ \\
&                                                                     & $= (2j+1,2j-1,2j,2j)$ \\ \cline{2-3}
     &  $\I^{(r-1)}(f, \{\pm 1\})$  cannot be replaced by  $\I^{(r)}(f, \{\pm 1\})$ & \lem{osneg}(ii)  \\
     \hline
\end{longtable}

\begin{longtable}{|c|c|c|}
 \caption{Copositive approximation} \label{coptable}  \\ \hline
\multicolumn{3}{|c|}{\bf Copositive approximation with interpolatory constraints ($f\in\ds$)}\\ \hline
 \multirow{4}{*}{$f\in C$}
          & \multirow{3}{*}{\parbox[t]{\twosize}{\centering $\exists$ $P_n\in \ds\cap \I(f, A_p)$:
             $|f(x)-P_n(x)| \leq c \w_2(f,\rho_n(x))$}} & \thm{thcopositive} with \\
          &                                             & $(r,k,m_1,m_2, m_3)$ \\
          &                                             & $=(0,2,0,0,0)$ \\   \cline{2-3}
   &   $\w_2(f,\rho_n(x))$ cannot be replaced by  $\w_3(f,\rho_n(x))$
 &   \lem{lemma2003new}  \\ \hline
 \multirow{15}{*}{$f\in C^1$}
                             & \multirow{3}{*}{\parbox[t]{\twosize}{\centering $\exists$ $P_n\in\ds\cap  \I(f, A_p)$:
                                                                $|f(x)-P_n(x)| \leq c  \rho_n(x)  \w_3(f',\rho_n(x))$}} & \thm{thcopositive} with \\
          &                                                     & $(r,k,m_1,m_2, m_3)$ \\
          &                                                     & $=(1,3,0,0,0)$ \\   \cline{2-3}
                             & \parbox[t]{\twosize}{\centering  $\w_3(f',\rho_n(x))$ cannot be replaced by  $\w_4(f',\rho_n(x))$   unless $A_p = \{\pm 1\}$}   & \multirow{2}{*}{\parbox[t]{\threesize}{\centering \lem{lemma205} }} \\  \cline{2-3}
                             &   $\I (f, A_p)$ cannot   be replaced by  $\I^{(1)}(f, A_p)$   &  \lem{lemma28}  \\ \cline{2-3}
                              & \multirow{3}{*}{\parbox[t]{\twosize}{\centering $\exists$ $P_n\in\ds\cap  \I(f, \{\pm 1\})$:
                                                        $|f(x)-P_n(x)| \leq c  \rho_n(x)  \w_k(f',\rho_n(x))$}} &  \thm{thcopositive} with \\
          &                                                                                                     & $(r,k,m_1,m_2, m_3)$ \\
          &                                                                                                      & $=(1,k,-1,0,0)$ \\   \cline{2-3}
                              &   $\I(f, \{\pm 1\})$  cannot   be replaced by  $\I^{(1)}(f, \{\pm 1\})$   &  \lem{lemma28}    \\ \cline{2-3}
                              &   $\exists$   $ P_n \in\ds\cap  \I^{(1)} (f, Y_s)$:  &  \multirow{2}{*}{\parbox[t]{\threesize}{\centering follows from a stronger statement below}} \\
                              & $|f(x)-P_n(x)| \leq c \rho_n(x) \w_2(f',\rho_n(x)) $  &    \\   \cline{2-3}
                            & \parbox[t]{\twosize}{\centering  $\w_2(f',\rho_n(x))$ cannot be replaced by  $\w_3(f',\rho_n(x))$}   & \lem{lemma40} \\  \cline{2-3}
                              &  \multirow{3}{*}{\parbox[t]{\twosize}{\centering   $\exists$   $ P_n \in\ds\cap  \I(f, A_p) \cap  \I^{(1)} (f, Y_s)$:
                                                                $|f(x)-P_n(x)| \leq c \rho_n(x) \w_2(f',\rho_n(x)) $ }}  &    \thm{thcopositive} with    \\
          &                                                                                                      & $(r,k,m_1,m_2, m_3)$ \\
          &                                                                                                      & $=(1,2,0,1,0)$ \\   \hline
 \multirow{11}{*}{$f\in C^2$}
                                &  \multirow{3}{*}{\parbox[t]{\twosize}{\centering  $\exists$   $P_n \in\ds\cap \I^{(1)}(f, A_p) \cap \I (f, Y_s)$:
                                  $|f(x)-P_n(x)| \leq c \rho_n^2(x) \w_k(f'',\rho_n(x)) $}} & \thm{thcopositive} with    \\
            &                                                                                                      & $(r,k,m_1,m_2, m_3)$ \\
            &                                                                                                      & $=(2,k,1,0,1)$ \\ \cline{2-3}
                                &  $\I^{(1)}(f, A_p)$  cannot be replaced by  $\I^{(2)}(f, A_p)$ & \lem{lemma30}    \\ \cline{2-3}
                               &  $\exists$   $P_n \in\ds \cap \I^{(1)} (f, Y_s)$:  &    \multirow{2}{*}{\parbox[t]{\threesize}{\centering follows from a stronger statement below}}   \\
                                & $|f(x)-P_n(x)| \leq c \rho_n^2(x) \w_3(f'',\rho_n(x)) $  &  \\ \cline{2-3}
                            & \parbox[t]{\twosize}{\centering  $\w_3(f'',\rho_n(x))$ cannot be replaced by  $\w_4(f'',\rho_n(x))$ }   &  \lem{lemmam22}  \\  \cline{2-3}
                            & \parbox[t]{\twosize}{\centering   $\I^{(1)} (f, Y_s)$ cannot be replaced with  $\I^{(2)} (f, Y_s)$}   &  \lem{march21}  \\  \cline{2-3}
                            &   \multirow{3}{*}{\parbox[t]{\twosize}{\centering  $\exists$   $P_n \in\ds\cap \I^{(1)}(f, A_p) \cap \I^{(1)} (f, Y_s)$:
                                $|f(x)-P_n(x)| \leq c \rho_n^2(x) \w_3(f'',\rho_n(x)) $ }} & \thm{thcopositive} with    \\
             &                                                                                                      & $(r,k,m_1,m_2, m_3)$ \\
             &                                                                                                      & $=(2,3,1,1,1)$ \\ \hline
 \multirow{9}{*}{\parbox[t]{\onesize}{\centering $f\in C^r$, $r\ge 3$ is odd}}
     &    $\exists$   $ P_n \in\ds\cap \I^{(r-2)}(f, A_p)\cap \I^{(r-1)}(f, Y_s)$: &  \multirow{2}{*}{\parbox[t]{\threesize}{\centering follows from a stronger statement below}}   \\
     &                       $|f(x)-P_n(x)| \leq c \rho_n^r(x) \w_k(f^{(r)},\rho_n(x)) $  &  \\ \cline{2-3}
     &  \parbox[t]{\twosize}{\centering $\I^{(r-2)}(f, A_p)$  cannot be replaced by  $\I^{(r-1)}(f, A_p)$ unless $A_p=\{\pm 1\}$} & \multirow{2}{*}{\parbox[t]{\threesize}{\centering \lem{lemma50} }}  \\ \cline{2-3}
     &  \parbox[t]{\twosize}{\centering $\I^{(r-1)}(f, Y_s)$  cannot be replaced by  $\I^{(r)}(f, Y_s)$} &  \lem{lemma30ys}   \\ \cline{2-3}
     & \multirow{3}{*}{\parbox[t]{\twosize}{\centering   $\exists$   $ P_n \in\ds\cap  \I^{(r-2)}(f, A_p)\cap \I^{(r-1)}(f, Y_s)  \cap  \I^{(r-1)}(f, \{\pm 1\})$:
                                                         $|f(x)-P_n(x)| \leq c \rho_n^r(x) \w_k(f^{(r)},\rho_n(x)) $}} & \thm{thcopositive} with    \\
              &                                                                                                      & $(r,k,m_1,m_2, m_3)$ \\
              &                                                                                                      & $=(2j+1,k,2j-1,2j,2j)$ \\ \cline{2-3}
     &  $\I^{(r-1)}(f, \{\pm 1\})$  cannot be replaced by  $\I^{(r)}(f, \{\pm 1\})$ & \lem{lemma28}  \\
     \hline
 \multirow{5}{*}{\parbox[t]{\onesize}{\centering $f\in C^r$, $r\ge 4$ is even}}
                         & \multirow{3}{*}{\parbox[t]{\twosize}{\centering   $\exists$   $P_n \in\ds\cap \I^{(r-1)}(f, A_p) \cap \I^{(r-2)}(f, Y_s)$:
                                                                 $|f(x)-P_n(x)| \leq c \rho_n^r(x) \w_k(f^{(r)},\rho_n(x)) $ }} & \thm{thcopositive} with    \\
              &                                                                                                      & $(r,k,m_1,m_2, m_3) = $ \\
              &                                                                                                      & $(2j,k,2j-1,2j-2,2j-1)$ \\ \cline{2-3}
&  $\I^{(r-1)}(f, A_p)$  cannot be replaced by  $\I^{(r)}(f, A_p)$ & \lem{lemma30}    \\ \cline{2-3}
&  $\I^{(r-2)}(f, Y_s)$  cannot be replaced by  $\I^{(r-1)}(f, Y_s)$  &   \parbox[t]{\threesize}{\centering \lem{lemma50ys} }  \\
\hline
\end{longtable}

For $Y_s  \in \Y_s$ and $A_p   \in \A_p$, recalling that it is always assumed that $Y_s \cap A_p = \emptyset$,
it is convenient to denote
\be \label{unionya}
Y_s \cup A_p =: \left\{\beta_i\right\}_{i=1}^{s+p} , \quad \text{with $-1 \le \beta_{s+p} < \beta_{s+p-1} < \dots < \beta_2 < \beta_1 \le 1$,}
\ee
and
\begin{align}   \label{dya}
d(Y_s, A_p) & := \min\left\{ |u-v| \st u,v \in Y_s \cup A_p \cup \{\pm 1\} \andd u\ne v \right\} \\ \nonumber
&=
\pmin\left\{ \beta_{i-1}-\beta_i \st 1\le i \le s+p+1 \right\},
\end{align}
where $\beta_0 := 1$, $\beta_{s+p+1}:= -1$, and
$\pmin(S)$ is the smallest positive number from the set $S$ of nonnegative reals.
In particular, $d(Y_s) := d(Y_s, \emptyset)$ and $d(A_p) := d(\emptyset, A_p)$.

It follows from \rem{remapril26} as well as
Lemmas~\ref{tmplem} and \ref{tmplemder} that, in general, dependence of various constants in theorems below  on $d(Y_s, A_p)$ cannot be removed.

\sect{Main results}

\subsection{Onesided and intertwining approximation with interpolatory constraints}

For convenience, we introduce the following notation:
\[
%\intset  :=
\intset_{\text{intertwining}} :=  \bigcup_{r\ge 1}   \intset_r ,
\]
where
\begin{align*}
 \intset_{r}   :=    \left\{ (r,m_1, m_2, m_3) \st \right.  & m_1 = r-1, \;  m_2 = r-2, \; m_3 =  r-1 \; \text{if $r$ is even; }  \\
  & \left.  m_1 = r-2, \;  m_2 = r-1, \;  m_3 =  r-1 \; \text{if $r$ is odd} \right\} .
\end{align*}
Hence,
\begin{align}  \label{defintertwining}
\intset_{\text{intertwining}} = &  \left\{ (2j,2j-1, 2j-2, 2j-1) \st j\in\N \right\} \\ \nonumber
& \cup \left\{(2j+1, 2j-1, 2j, 2j) \st j\in\N_0 \right\} .
\end{align}
Note that, when $r=1$,
\[
%\Upsilon_0 := \{(0,-1,-1,-1)\}  \andd
 \intset_1 := \left\{(1,-1, 0, 0) \right\} ,
\]
which is the only case when any of $m_i$'s can become negative.

 The following is our main result (or rather a collection of results combined in one statement) for intertwining approximation with interpolatory constraints.

\begin{theorem}[intertwining  approximation with interpolatory constraints]\leavevmode   \label{th:inti}%
Let $k,p,s\in\N$,    $Y_s \in\Y_s$, $A_p\in\A_p$, and $(r,m_1,m_2, m_3)\in\intset_{\text{intertwining}}$, %with $\intset_{\text{intertwining}}$
defined in \ineq{defintertwining}.
Then, for any  $f\in C^r$ and any $n\ge c(d(Y_s, A_p))$, there exists
\[
  P_n \in \Pn\cap \td\cap  \I^{(m_1)}(f, A_p)\cap \I^{(m_2)}(f, Y_s)  \cap  \I^{(m_3)}(f, \{\pm 1\})
\]
 such that
\[
|f(x)-P_n(x)| \leq   c(k,r,s,p) \rho_n^r(x) \w_k(f^{(r)},\rho_n(x)), \qquad x\in I,
\]
where $\I^{(-1)}(f, A_p) := \left\{ h: I \mapsto \R\right\}$, the set of all real valued functions on $I$,  \ie the restriction $P_n\in \I^{(-1)}(f, A_p)$ is vacuous, and so
 there is no requirement that $P_n$ interpolate $f$ at the points in $A_p$  if $m_1=-1$.
\end{theorem}

The results for onesided approximation are similar and could have been included in the statement of \thm{th:inti} by allowing $s$ to be $0$ and making the restriction
 $P_n\in \I^{(m_2)}(f, Y_0)$ vacuous. Nevertheless, in order to avoid possible confusion and for readers' benefit, we provide an explicit statement.

 Let
\begin{align} \label{defonesided}
\intset_{\text{onesided}}   := &  \left\{(r,m_1, m_3) \st  (r,m_1, m_2, m_3) \in \intset_{\text{intertwining}}, \; \text{for some $m_2\ge 0$} \right\}  \\ \nonumber
    =    & \left\{ (2j, 2j-1, 2j-1) \st j\in\N  \right\} \cup   \left\{ (2j+1, 2j-1, 2j) \st j\in\N_0 \right\} .
\end{align}

 \begin{theorem}[onesided  approximation with interpolatory constraints]\leavevmode\label{thone}%
Let $k,p\in\N$,   $A_p\in\A_p$, and $(r,m_1,  m_3)\in\intset_{\text{onesided}}$, defined in \ineq{defonesided}.
Then, for any  $f\in C^r$ and any $n\ge c(d(A_p))$, there exists
\[
  P_n \in  \Pn \cap   \I^{(m_1)}(f, A_p)   \cap  \I^{(m_3)}(f, \{\pm 1\})
\]
 such that $P_n \ge f$ on $I$, and
\[
|f(x)-P_n(x)| \leq   c(k,r, p) \rho_n^r(x) \w_k(f^{(r)},\rho_n(x)), \qquad x\in I.
\]
Here, as in the statement of \thm{th:inti},
 there is no requirement that $P_n$ interpolate $f$ at the points in $A_p$  if $m_1=-1$.
\end{theorem}

It follows from Lemmas~\ref{osneg} and \ref{interneg}
 that Theorems~\ref{th:inti} and \ref{thone}  cannot be strengthened by increasing any of the parameters $m_i$'s (for  any $r\in\N$), or by including the case for  $r=0$.

\subsection{Positive and copositive approximation with interpolatory constraints}

Similarly to the previous section, we start with introducing the   notation for the set of indices $\copset_{\text{copositive}}$. However, in order to capture all statements in Table~\ref{coptable}, this notation needs to be more complicated and involve the order of the modulus $k$. We denote
\[
%\intset  :=
\copset_{\text{copositive}} :=  \bigcup_{r\ge 0}   \copset_{r} ,
\]
where
\begin{align*}
\copset_{0} & := \{ (0,2, 0, 0, 0) \} , \\
 \copset_{1} &  := \{ (1,3, 0, 0, 0), (1,2,0, 1, 0) \}    \cup \left\{ (1,k,-1, 0, 0) \st k\in\N \right\}             , \\
 \copset_{2}  & :=  \{ (2,3,1, 1, 1) \}  \cup \left\{ (2,k, 1, 0, 1)\st k\in\N \right\} ,
\end{align*}
and, for $r\ge 3$,
\begin{align*}
 \copset_{r}     :=   \left\{ (r,k, m_1, m_2, m_3) \st \right. & \left.  k\in\N \andd  (r,  m_1, m_2, m_3)\in \intset_r \right\} \\
=    \left\{ (r,k, m_1, m_2, m_3) \st \right.  & k\in\N,  \;\; \text{and}\;\; m_1 = r-1, \;  m_2 = r-2, \; m_3 =  r-1 \; \text{if $r$ is even; }  \\
  & \left.  m_1 = r-2, \;  m_2 = r-1, \;  m_3 =  r-1 \; \text{if $r$ is odd} \right\} .
\end{align*}
Hence,
\begin{align} \label{defcopositive}
\copset_{\text{copositive}} = & \left\{ (0,2, 0, 0, 0), (1,3, 0, 0, 0), (1,2,0, 1, 0), (2,3,1, 1, 1) \right\} \\ \nonumber
& \cup
 \left\{ (2j,k, 2j-1, 2j-2, 2j-1) \st j,k\in\N \right\} \\ \nonumber
 &  \cup \left\{(2j+1, k, 2j-1, 2j, 2j) \st j\in\N_0,  \; k\in\N\right\} .
\end{align}

\begin{theorem}[copositive   approximation with interpolatory constraints]\leavevmode\label{thcopositive}%
Let $p,s\in\N$,    $Y_s \in\Y_s$, $A_p\in\A_p$, and $(r,k,m_1,m_2, m_3)\in\copset_{\text{copositive}}$, defined in \ineq{defcopositive}.
Then, for any  $f\in C^r\cap  \Delta^{(0)}(Y_s)$ and any $n\ge c(d(Y_s, A_p))$, there exists
\[
  P_n \in \Pn\cap \Delta^{(0)}(Y_s) \cap  \I^{(m_1)}(f, A_p)\cap \I^{(m_2)}(f, Y_s)  \cap  \I^{(m_3)}(f, \{\pm 1\})
\]
 such that
\[
|f(x)-P_n(x)| \leq   c(k,r,s,p) \rho_n^r(x) \w_k(f^{(r)},\rho_n(x)), \qquad x\in I.
\]
Here, as in the statement  of \thm{th:inti},
 there is no requirement that $P_n$ interpolate $f$ at the points in $A_p$  if $m_1=-1$.
\end{theorem}

As in the case for intertwining and onesided approximation, the results for positive  approximation are similar to those in the copositive case and could have been included in the statement of \thm{thcopositive} by allowing $s$ to be $0$ and making the restriction
 $P_n\in \I^{(m_2)}(f, Y_0)$ vacuous. Again, we provide an explicit statement for readers' convenience.

Let
\begin{align} \label{defpositive}
\copset_{\text{positive}}   := &  \left\{(r,k, m_1, m_3) \st  (r,k, m_1, m_2, m_3) \in \copset_{\text{copositive}}, \; \text{for some $m_2\ge 0$} \right\}  \\ \nonumber
    =    &
    \left\{ (0,2, 0,  0), (1,2, 0,  0), (1,3, 0,  0),  (2,3,1,   1) \right\}  \\ \nonumber
   &  \cup
    \left\{ (2j, k,  2j-1, 2j-1) \st j,k\in\N  \right\} \\ \nonumber
    &  \cup   \left\{ (2j+1, k,  2j-1, 2j) \st j \in\N_0, \; k\in\N \right\}  .
\end{align}

\begin{theorem}[positive   approximation with interpolatory constraints]\leavevmode\label{thpositive}%
Let $p \in\N$,   $A_p\in\A_p$, and $(r,k,m_1,  m_3)\in\copset_{\text{positive}}$, defined in \ineq{defpositive}.
Then, for any  $f\in C^r$ such that $f\ge 0$ on $I$, and any $n\ge c(d(A_p))$, there exists
\[
  P_n \in \Pn\cap  \I^{(m_1)}(f, A_p)  \cap  \I^{(m_3)}(f, \{\pm 1\})
\]
 such that $P_n \ge 0$ on $I$, and
\[
|f(x)-P_n(x)| \leq   c(k,r,p) \rho_n^r(x) \w_k(f^{(r)},\rho_n(x)), \qquad x\in I.
\]
Here, as in the statements above,
 there is no requirement that $P_n$ interpolate $f$ at the points in $A_p$  if $m_1=-1$.
\end{theorem}

It follows from lemmas in Sections~\ref{section3.2} and  \ref{section3.3}
 that Theorems~\ref{thcopositive} and \ref{thpositive}  cannot be strengthened by increasing any of the parameters $k$ or $m_i$'s, for  any $r\ge 0$.

 %\ref{lemma23},  \ref{lemma2003new},    \ref{lemma30}, \ref{lemma205}, \ref{lemma28}, \ref{lemma50},  \ref{lemma40}, \ref{lemmam22}, \ref{march21}, \ref{lemma30ys} and \ref{lemma50ys}

\sect{Negative results}

Recall that, for functions $g: I\mapsto \R$ with absolutely continuous $(r-1)$st derivatives on $I$, the following (Taylor's formula) holds:
\[
g(x) =  \sum_{i=0}^{r-1}  \frac{1}{i!}  g^{(i)}(a)(x-a)^{i}  +   \frac{1}{(r-1)!} \int_a^x (x-t)^{r-1} g^{(r)}(t) dt , \quad x\in I.
\]
We often use this elementary identity   without explicitly mentioning it.

\subsection{Onesided and intertwining approximation with interpolatory constraints}

 The purpose of the following lemmas is to establish   general negative results  for onesided and intertwining approximation with interpolatory constraints.

\begin{lemma}[local onesided approximation with interpolatory constraints: negative result for $f\in C^r$ with $r\ge 0$] \leavevmode \label{osneg}
Let $r\in\N_0$. Then,
\begin{enumerate}[\rm (i)]
\item  if $\lambda$ is an interior point of $I$ (\ie $\lambda\in (-1,1)$),
then there exists a function $f\in C^r$   such that $f \equiv 0$ on $[-1, \lambda]$, and
\[
\I^{(\tr)}(f, \{\lambda\}) \cap \left\{ g \in C^{r+1}  \st \text{$g  \geq f$ on $(\lambda-\e,\lambda+\e)$,  for some $\e>0$} \right\} = \emptyset ,
\]
where
\[
\tr  :=2 \lfloor r/2 \rfloor =
\begin{cases}
r-1 \,, & \text{if $r$ is odd,} \\
r  \,, & \text{if $r$ is even.}
\end{cases}
\]
\item
 if $\lambda$ is an endpoint of $I$ (\ie $\lambda = \pm 1$), then there exists   $f\in C^r$   such that
\[
\I^{(r)}(f, \{\lambda\}) \cap \left\{ g \in C^{r+1}  \st \text{$g  \geq f$ on $(\lambda-\e,\lambda+\e)\cap I$,  for some $\e>0$} \right\} = \emptyset .
\]
\end{enumerate}
\end{lemma}

\begin{proof}
If $\lambda$ is an interior point of $I$ then, without loss of generality, we assume that  $\lambda=0$ and
 define
\[
f(x) =     \frac{1}{(r-1)!} \int_0^x (x-t)^{r-1} f^{(r)}(t) dt , \quad \text{where} \quad
%\]
%where
%\[
f^{(r)}(x) :=
\begin{cases}
-x \ln x , & \text{if $0<x\le 1$,}\\
0,   & \text{if $-1\le x\le 0$.}
\end{cases}
\]
Clearly, $f\in C^r$ and $f \equiv 0$ on $[-1, 0]$.  Now, suppose that, for some $\e>0$, there exists $g \in C^{r+1}$ satisfying $g(x) \geq f(x)$, $x\in (-\e,\e)$, and such that $g^{(i)}(0)=f^{(i)}(0)=0$, $0\le i \le \tr$.
Then, since $\tr \ge r-1$,
\begin{eqnarray*}
g(x)-f(x) &=&  \frac{1}{(r-1)!}\int_0^x (x-t)^{r-1} \left[ g^{(r)}(t) -f^{(r)}(t) \right]  dt .
\end{eqnarray*}

If $r$ is even, then $\tr=r$, and so $g^{(r)}(0) = 0$, and we will now show that this equality also holds if $r$ is odd.
Indeed, for $-\e< x<0$,
\[
0\le g(x)-f(x) = \frac{(-1)^r}{(r-1)!} \int_x^0 (t-x)^{r-1} g^{(r)}(t) dt ,
\]
and so, by continuity of $g^{(r)}$, we must have $(-1)^r g^{(r)}(0) \ge 0$. Similarly, because $g \ge f$ on $(0,\e)$, we should  have  $g^{(r)}(0) \ge f^{(r)}(0)= 0$.  Hence, if  $r$ is odd, this implies that  $g^{(r)}(0) = 0$.

Thus, for $x>0$, we have
\begin{eqnarray} \label{mar11}
g(x)-f(x) &=& \frac{1}{(r-1)!}\int_0^x (x-t)^{r-1} \left[ g^{(r)}(t) +t \ln t \right]  dt  \\ \nonumber
& =&  \frac{1}{(r-1)!}\int_0^x (x-t)^{r-1} \int_0^t \left( g^{(r+1)}(u) + \ln u +1 \right) \, du \, dt .
\end{eqnarray}
 Since $g\in C^{r+1}$ and $\lim_{u\to 0^+}  \ln u = -\infty$, there exists $\delta>0$ such that $g^{(r+1)}(u) + \ln u +1 <0$, for all $0<u<\delta$. Hence,
$g(x)-f(x) <0$, for $0<x<\delta$,
which is a contradiction.

\medskip

If $\lambda$ is an endpoint of $I$, then the same proof as above works with no changes (for example, one can consider $[0,1]$ instead of $I$, $\lambda=0$, and then arrive at a contradiction using \ineq{mar11}).
\end{proof}

The same proof can be used to verify validity of the following lemma that is used for negative results involving  Hermite interpolation at the points in $Y_s$.

\begin{lemma}[local intertwining approximation with interpolation at  $Y_s$: negative result  for $f\in C^r$ with $r\ge 1$] \leavevmode \label{interneg}
Let $r\in\N$
 and suppose that   $\lambda\in (-1,1)$.
Then there exists a function $f\in C^r$      such that $f \equiv 0$ on $[-1, \lambda]$, and
\[
\I^{(\ttr)}(f, \{\lambda\}) \cap \left\{ g \in C^{r+1}  \st \text{$g  \leq f$ on $(\lambda-\e, \lambda]$, and $g  \geq f$ on $[\lambda, \lambda+\e)$,   for some $\e>0$} \right\} = \emptyset ,
\]
where
\[
\ttr  := 2 \lceil  r/2 \rceil -1 =
\begin{cases}
r  \,, & \text{if $r$ is odd,} \\
r-1  \,, & \text{if $r$ is even.}
\end{cases}
\]
\end{lemma}

 We remark that in all cases that are not covered by the statements of Lemmas~\ref{osneg} and \ref{interneg}  onesided/intertwining   polynomial approximation with interpolatory constraints is possible (see Theorems~\ref{th:inti} and \ref{thone}).

Note that the only reason for including the property that $f \equiv 0$ on $[-1, \lambda]$ in the statements of
\lem{osneg} (if $\lambda$ is an interior point of $I$)  and \lem{interneg}
 is to indicate  that these lemmas are  still valid if $f$ is assumed to be from the class   $\Delta^{(0)}(Y_s)$ (with $y_1 \le \lambda$). In particular, this is important for the illustrative example in Section~\ref{sectionpc}.

\subsection{Positive and copositive approximation with interpolatory constraints} \label{section3.2}

The following is the first lemma in this paper that provides a negative result for {\em polynomial} approximation with interpolatory constraints. Even though, in some sense, it is less general than \lem{lemma2003new} that follows it, its proof is rather simple and illustrates the main ideas that are used in the proofs of  more complicated negative results below but without the  technical details required there.

\begin{lemma}[positive approximation with interpolatory constraints: negative result for $f\in C$]\leavevmode\label{lemma23}
For any $n\in\N$,  $A>0$ and $\lambda\in I$, there exists $f\in C$ which is nonnegative on $I$ and such that, for any nonnegative on $I$ polynomial $P_n \in \Pn$ satisfying $P_n(\lambda)=f(\lambda)$, we have
\[
\norm{f-P_n}{} > A \w_3(f, 1) .
\]
\end{lemma}

\begin{proof}%[Proof of \lem{lemma23}]
Let $n\in \N$ and $A>0$ be fixed, and without loss of generality suppose that $\lambda \in[0,1]$.
We now let $\e \in (0,  \lambda/2)$, and pick $f(x) := \max\left\{ (x-\lambda)(x-\lambda+\e)/\e^2, 0 \right\}$ and suppose that $P_n \in \Pn$ is nonnegative on $I$, satisfies  $P_n(\lambda)=f(\lambda) = 0$ (this necessarily implies that $P_n'(\lambda)\le 0$, and actually $P_n'(\lambda)=0$ if $\lambda \ne 1$) and such that
$\norm{f-P_n}{} \le  A \w_3(f, 1)$.
If $Q(x) := (x-\lambda)(x-\lambda+\e)/\e^2$, then
$\norm{f-Q}{} = \norm{Q}{[\lambda-\e,\lambda]} = 1/4$
and
\[
\w_3(f, 1) = \w_3(f-Q, 1) \le 8 \norm{f-Q}{} = 2.
\]
Hence,
\[
\norm{P_n-Q}{} \le   \norm{P_n-f}{} +  \norm{f-Q}{} \le 2A+1/4 ,
\]
and so, by Markov's inequality,
\[
1/\e =  Q'(\lambda) \le   Q'(\lambda)-P_n'(\lambda)  \le \norm{Q'-P_n'}{}  \le (2A+1/4) (\max\{n,2\})^2.
\]
We now get a contradiction by letting $\e \to 0^+$.
\end{proof}

\begin{lemma}[copositive approximation with interpolatory constraints: negative result for $f\in C$]\leavevmode \label{lemma2003new}
For any  $A>1$,  $\lambda\in [0,1]$, $n_0\in\N$, and any positive sequence $(\alpha_n)_{n\in\N}$ converging to $0$, there exist  $n\in\N$,  $n\ge n_0$,  and $f = f_n \in C$  such that
 $f\ge 0$ on $I$,
  $f \equiv 0$ on $[-1,-1/2]$,
 and for  any  polynomial $P_n \in \Pn$ which is nonnegative in some nonempty neighborhood of $\lambda$
  and  satisfies $P_n(\lambda)=f(\lambda)$, we have
\[
\left\| \frac{f-P_n}{  \w_3(f, \rho_n)}\right\|  > A \, \alpha_n  \sqrt{n} .
\]
\end{lemma}

\begin{proof}
Recall that the classical ``glue function'' is
\[
\tG(x):=  \tG (x; [a,b]) := \frac{G(x; [a,b])}{G(1; [a,b])}    , \quad \text{where} \quad                   G(x; [a,b]):=   \int_{-1}^x g(t; (a+b)/2, (b-a)/2) \, dt ,
\]
\[
g(x; x_0, d) :=
\begin{cases}
\exp(d^2/( (x-x_0)^2-d^2)) , &  x\in (x_0-d, x_0+d) , \\
0 , & \text{otherwise,}
\end{cases}
\]
and $[a,b]\subset [-1,1]$. Then, $\tG\in C^\infty$, $\tG (x) =0$ if $x\le a$,  $\tG (x) =1$ if $x\ge b$, and $0\le \tG (x) \le 1$ if $a\le x\le b$.

Now,   let  $\e := \rho_n^{3/2}(\lambda)$,  $Q(x) := (x-\lambda)(x-\lambda+\e)$,  and define
\[
f(x) :=
\begin{cases}
0 , & x\in [\lambda-\e, \lambda], \\
Q(x)  , & x\in [-1/4,1]\setminus [\lambda-\e, \lambda] ,\\
Q(x) \, \tG (x; [-1/2,-1/4]) , & x\in (-1/2, -1/4), \\
0, & x\in [-1,-1/2] .
\end{cases}
\]
(Here, we assume that $n$ is sufficiently large so that everything is well defined.)

Suppose now that
$P_n \in \Pn$
  is nonnegative near $\lambda$, satisfies  $P_n(\lambda)=f(\lambda) = 0$ (this necessarily implies that $P_n'(\lambda) \le  0$, and actually $P_n'(\lambda)=0$ if $\lambda \ne 1$) and such that
$|f(x)-P_n(x)| \le  A \alpha_n \sqrt{n} \, \w_3(f, \rho_n(x))$, $x\in I$.

Clearly, $f\in C^\infty[-1,-1/8]$, and it is not difficult to see that $\norm{f^{(3)}}{[-1,-1/8]}\le c_0$, for some absolute constant $c_0$ (it is convenient to assume that $c_0\ge 2$). Hence,
\[
\w_3(f, t; [-1,-1/8]) \le  t^3  \norm{f^{(3)}}{[-1,-1/8]} \le c_0 t^3  , \quad t>0.
\]
Since
\[
\norm{f-Q}{[-1/4,1]} = \norm{Q}{[\lambda-\e,\lambda]} = \e^2/4 ,
\]
 we also have
\[
\w_3(f, t; [-1/4,1]) = \w_3(f-Q, t; [-1/4,1]) \le 8 \norm{f-Q}{[-1/4,1]}   = 2 \e^2 .
\]
Hence,
\[ %\be \label{modulus}
\w_3(f, \rho_n(x))  \le  c_0 \max\{ \rho_n^3(x),  \rho_n^3(\lambda) \}, \quad x\in I.
\]
For $x\in [-1/4,1]$, we now have
 \be \label{in1}
 |P_n(x)-Q(x)| \le  |P_n(x)-f(x)|+|f(x)-Q(x)| \le    (c_0 A \alpha_n \sqrt{n}    + 1) \max\{ \rho_n^3(x),  \rho_n^3(\lambda) \}.
 \ee

We now consider two cases.

\medskip

\noindent
{\bf Case 1:} $\lambda < 1$  \\
Since $\rho_n(x) \le 2 n^{-1}$, $x\in I$,
using \ineq{in1} and  Bernstein's inequality on $[-1/4,1]$, we have
\begin{align*}
(1-\lambda^2)^{3/4} n^{-3/2} & \le  \e   =  Q'(\lambda) =   |Q'(\lambda)-P_n'(\lambda)|  \le   (1-\lambda)^{-1/2}(\lambda+1/4)^{-1/2} n \norm{Q-P_n}{[-1/4,1]} \\
&   \le   16  (1-\lambda)^{-1/2}       n^{-2} \left(  c_0 A \alpha_n    \sqrt{n} + 1\right),
\end{align*}
and we arrive at a contradiction by taking $n$ to be sufficiently large.

\medskip

\noindent
{\bf Case 2:} $\lambda = 1$  \\
We need the following Dzyadyk inequality (see \eg   \cite{dzya}*{p. 386} or \cite{kls21}*{Lemma 2.5}):
\begin{quote}
If $m\in\N$,  $x_0\in I$, $R_n\in\Pn$, and
\[
|R_n(x)|\le(|x-x_0|+\rho_n(x))^m,\quad x\in I,
\]
then
\[
|R_n'(x_0)|\le c(m) \rho_n^{m-1}(x_0) .
\]
\end{quote}
We use this result with $m=3$,  $x_0=\lambda=1$  and $R_n = B^{-1}(P_n-Q)$, where $B:=  c_0 A \alpha_n    \sqrt{n} + 1$.
 Suppose that we verified that
\be \label{nts}
|P_n(x)-Q(x)|\le B (|x-1|+\rho_n(x))^3,\quad x\in I.
\ee
Then, by Dzyadyk's inequality,
\[
 n^{-3}  =  \e   =  Q'(1)  \le    Q'(1)-P_n'(1)  \le  c B \rho_n^{2}(1) = c (c_0 A \alpha_n    \sqrt{n} + 1) n^{-4} ,
\]
and we arrive at a contradiction by taking $n$ to be sufficiently large.

It remains to verify \ineq{nts}. Note that \ineq{nts} immediately follows from \ineq{in1} for $x\in [-1/4, 1]$ and, for $x\in [-1,-1/4]$, we have
\[
|Q(x)| \le x(x-1) \le (1-x)^3,
\]
and so
\begin{align*}
|P_n(x)-Q(x)| & \le  |P_n(x)-f(x)|+|f(x)-Q(x)| \le |P_n(x)-f(x)| + |Q(x)| \\
& \le c_0 A \alpha_n \sqrt{n} \rho_n^3(x) + |x-1|^3
\le
B (|x-1|+\rho_n(x))^3 .
\end{align*}
The proof is now complete.
\end{proof}

 \begin{lemma}[copositive approximation with interpolatory constraints: negative result for $f\in C^1$ and $\lambda\ne 1$]\leavevmode \label{lemma205}
 For any  $A>1$,  $\lambda\in [0,1)$, $n_0\in\N$,  and any positive sequence $(\alpha_n)_{n\in\N}$ converging to $0$, there exist  $n\in\N$, $n\ge n_0$,  and $f=f_n \in C^1$  such that
 $f\ge 0$ on $I$,
  $f \equiv 0$ on $[-1,-1/2]$,
 and for any polynomial $P_n \in \Pn$ which is nonnegative on   $(\lambda-1/n, \lambda+1/n)$
  and  satisfies $P_n(\lambda)=f(\lambda)$, we have
\be \label{march6}
\norm{f-P_n}{} > A \,   \alpha_n n^{-2/3}  \w_4(f', n^{-1}) .
\ee
\end{lemma}

We remark that inequality \ineq{march6} implies that
\[
\left\| \frac{f-P_n}{ \rho_n \w_4(f', \rho_n)}\right\|  > c A \,  \alpha_n n^{1/3} ,
\]
where $c>0$ is some  absolute constant.  Also, note that the statement of this lemma is no longer valid if $\lambda=1$, \ie if  interpolation takes place at an endpoint of $I$
(see \thm{thcopositive} with  $(r,k,m_1,m_2, m_3) =(1,k,-1,0,0)$).

\begin{proof}[Proof of \lem{lemma205}]
Suppose that $n\ge 25(1-\lambda)^{-1}$, let   $\e := n^{-4/3}$ and  define
\[
g(x) :=
\begin{cases}
Q(x)   , & \text{if $x<\lambda$,} \\
0, & \text{if $\lambda \le x \le \lambda+\e $,} \\
Q(x)  +\e^4  , & \text{if $x \ge \lambda+\e $.}
\end{cases}
\]
where $Q(x) := 3(x-\lambda)^4-4\e (x-\lambda)^3$.
Note that $g\in C^1$ and
\[
g'(x) =
\begin{cases}
Q'(x)  , & \text{if $x<\lambda $ or $x>\lambda+\e$,} \\
0, & \text{if $\lambda \le x \le \lambda+\e $,}
\end{cases}
\]
with  $Q'(x)  = 12 (x-\lambda)^2 (x-\lambda-\e)$.

Now, define
\[
f (x) :=
\begin{cases}
g(x)  , & x\in [-1/4, 1], \\
g(x) \, \tG (x; [-1/2,-1/4]) , & x\in (-1/2, -1/4), \\
0, & x\in [-1,-1/2] ,
\end{cases}
\]
where $\tG$ is the glue function from the proof of \lem{lemma2003new}.

Suppose that $P_n \in \Pn$ is nonnegative on $(\lambda-1/n, \lambda+1/n)$ , satisfies  $P_n(\lambda)=f(\lambda) = 0$ (this necessarily implies that $P_n'(\lambda)= 0$) and such that
$\norm{f-P_n}{} \le  A \alpha_n  n^{-2/3} \, \w_4(f', 1/n)$.

Clearly, $f\in C^\infty[-1,-1/8]$, and it is not difficult to see that $\norm{f^{(5)}}{[-1,-1/8]}\le c_0$, for some absolute constant $c_0\ge 32$. Hence,
\[
\w_4(f', n^{-1}; [-1,-1/8]) \le   n^{-4} \norm{f^{(5)}}{[-1,-1/8]} \le c_0 n^{-4}.
\]
Since
\[
\norm{f'-Q'}{[-1/4,1]} = \norm{Q'}{[\lambda,\lambda+\e]} \le  2 \e^3 ,
\]
 we also have
\[
\w_4(f', n^{-1}; [-1/4,1]) = \w_4(f'-Q', n^{-1}; [-1/4,1]) \le 16 \norm{f'-Q'}{[-1/4,1]}   \le  32 \e^3 .
\]
Hence,
\[
\w_4(f', n^{-1}) \le  \max\{c_0 n^{-4}, 32\e^3 \} = c_0 n^{-4} ,
\]
and so
\begin{align*}
\norm{P_n-Q}{[-1/4,1]} & \le  \norm{P_n-f}{[-1/4,1]} +  \norm{g-Q}{[-1/4,1]} \\
& \le  A \alpha_n  n^{-2/3} \, \w_4(f', 1/n) +  \max\left\{ \e^4, \norm{Q}{[\lambda,\lambda+\e]} \right\}      \\
 &   \le    n^{-14/3} \left(  c_0 A \alpha_n   +  n^{-2/3} \right) .
% &   \le     c_0 A \alpha_n n^{-14/3} +  n^{-16/3}.
\end{align*}
By the Markov-Bernstein  inequality on $[-1/4,1]$,
\[
\norm{P_n'' -Q''}{[\lambda, \lambda+\e]}  \le c   (1-\lambda)^{-1} n^{-8/3} \left(  c_0 A \alpha_n + n^{-2/3} \right) ,
\]
and taking into account that $P_n(\lambda+\e) \ge 0$ and  $P_n^{(i)}(\lambda)=Q^{(i)}(\lambda)=0$, $i=0,1$, we have
\begin{align*}
1 & \le \e^{-4} \left[ P_n(\lambda+\e)-Q(\lambda+\e)\right]  =  \e^{-4}  \int_\lambda^{\lambda+\e} (\lambda+\e -t)  \left[P_n''(t) -Q''(t)\right] dt \\
& \le c \e^{-2}  (1-\lambda)^{-1} n^{-8/3} \left(  c_0 A \alpha_n + n^{-2/3} \right) \\
& \le c (1-\lambda)^{-1}   \left(  c_0 A \alpha_n + n^{-2/3} \right) ,
\end{align*}
and we arrive at a contradiction by taking $n$ to be sufficiently large.
\end{proof}

\begin{lemma}[copositive approximation with interpolatory constraints: negative result for $f\in C^r$ with even $r\ge 2$]\leavevmode \label{lemma30}
If $r\ge 2$ is \underline{even}, then for any $n\ge r$,  $A>0$ and $\lambda\in [0,1]$, there exists $f\in C^r$ which is nonnegative on $I$, is identically $0$ on $[-1,-1/2]$,
and such that, for any nonnegative on $(\lambda-1/n,\lambda)$  polynomial $P_n \in \Pn$ satisfying $P_n^{(i)}(\lambda)=f^{(i)}(\lambda)$, $0\le i \le r$,  we have
\be \label{march8}
\norm{f-P_n}{} > A \norm{f^{(r)}}{} .
\ee
\end{lemma}

 \begin{proof}
 Let $n\ge r$ and $A\ge 1$ be fixed, and define $f\in C^r$ as
\[
f(x) =     \frac{1}{(r-1)!} \int_{\lambda-\e}^x (x-t)^{r-1} f^{(r)}(t) dt ,
\]
 where
 \[
f^{(r)}(x) :=
\begin{cases}
0 , & \text{if $x\le \lambda-\e$ or $x\ge \lambda$,} \\
-(x-\lambda)(x-\lambda+\e)/\e^2  , & \text{if $\lambda-\e <  x <  \lambda $,}
\end{cases}
\]
and $\e\in(0,1/(2n))$. Then, $f\equiv 0$ on $[-1,\lambda-\e]$, $f\uparrow$ on $[\lambda-\e, 1]$
(and so $f$ is nonnegative on $I$), $\norm{f^{(r)}}{} = 1/4$ and, for $x\ge \lambda$,
\be  \label{apr29}
f(x)   = Q(x)   :=  \frac{1}{(r-1)!} \cdot \frac{1}{\e^2} \int_{\lambda-\e}^\lambda  (x-t)^{r-1} (\lambda-t)(t-\lambda+\e)   dt . % \\
%& = \frac{1}{(r+1)!} \cdot \frac{1}{\e^2} \left[ \e(x-\lambda+\e)^{r+1} + \e (x-\lambda)^{r+1} - \frac{2}{r+2} \left( (x-\lambda+\e)^{r+2}-(x-\lambda)^{r+2} \right)  \right] .
\ee
Note that $Q$ is a polynomial of degree $\le r-1$ and
\[
 Q(\lambda-\e) =    \frac{(-1)^{r-1}}{(r-1)!} \cdot \frac{1}{\e^2} \int_{\lambda-\e}^\lambda  (t-\lambda+\e)^{r} (\lambda-t)    dt = (-1)^{r-1}\frac{r}{(r+2)!} \e^r .
\]
Now, suppose that $P_n \in\Pn$ is such that $P_n\ge 0$ on $(\lambda-1/n,\lambda)$,  $P_n^{(i)} (\lambda) = f^{(i)}(\lambda)=Q^{(i)}(\lambda)$, $0\le i \le r$, and
$\norm{f-P_n}{} \le  A \norm{f^{(r)}}{} =A/4$.
Hence,
\[
\norm{P_n}{[-1,\lambda]} \le \norm{P_n-f}{[-1,\lambda]} + \norm{f}{[-1,\lambda]} \le A\norm{f^{(r)}}{}+ \norm{f}{[\lambda-\e,\lambda]} \le (A+\e^r/r!) \norm{f^{(r)}}{} \le A .
\]
Now, by Markov's inequality,
\[
\norm{P_n^{(r+1)}}{[-1, \lambda]} \le \left(\frac{2n^2}{\lambda+1}\right)^{r+1}   \norm{P_n}{[-1, \lambda]} \le (2n^2)^{r+1} A,
\]
and so, since $Q^{(r+1)}\equiv 0$ and $r$ is even,
\begin{align*}
\frac{r}{(r+2)!} \e^r  &=
-Q(\lambda-\e) \le  P_n(\lambda-\e)  -Q(\lambda-\e) \\
&=
    \frac{1}{r!} \int_\lambda^{\lambda-\e} (\lambda-\e-t)^{r} \left( P_n^{(r+1)}(t)  -Q^{(r+1)}(t) \right)   dt  \\
& \le \frac{1}{(r+1)!} \e^{r+1} \norm{P_n^{(r+1)}}{[\lambda-\e, \lambda]}
\le \frac{(2n^2)^{r+1} A}{(r+1)!} \e^{r+1} ,
\end{align*}
and we get a contradiction by letting $\e\to 0^+$.
\end{proof}

\begin{remark} \label{remapril26}
The same proof can be used to show that the estimate \ineq{march8} holds if $\lambda\in [0,1)$ and     $P_n\in\Pn$, such that $P_n\ge 0$ on $(\lambda-1/n,\lambda)$, is assumed to satisfy \eg
$P_n^{(i)}(\lambda)=f^{(i)}(\lambda)$, $0\le i \le r-1$, and $P_n (\lambda+ \e)=f (\lambda+ \e)$. In particular, this implies that constants $c(d(A_p))$ in Theorems~\ref{thcopositive} and \ref{thpositive}
(and, hence, in Theorems~\ref{th:inti} and \ref{thone})
 cannot be made independent of the minimal distance among points in $A_p$.
\end{remark}

\begin{lemma}[copositive approximation with interpolatory constraints: negative result for $f\in C^r$ with odd $r\ge 3$]\leavevmode \label{lemma50}
If $r\ge 3$ is \underline{odd}, then for
 any $n\ge r-1$ and $A>0$, there exists $f\in C^r$
 which is identically $0$ on $[-1,-1/2]$, is nonnegative on $I$ and such that, for any nonnegative on $[-1/2,1/2]$ polynomial $P_n \in \Pn$ satisfying $P_n^{(i)}(0)=f^{(i)}(0)$, $0\le i \le r-1$, we have
\[
\norm{f-P_n}{} > A \norm{f^{(r)}}{} .
\]
\end{lemma}

We note that an analog of  this lemma is not valid if we   require that $P_n$ be  nonnegative on $[-1/2,1]$ and  $P_n^{(i)}(1)=f^{(i)}(1)$, $0\le i \le r-1$ (\ie this negative result  does not hold if $A_p = \{\pm 1\}$),
see \eg
\thm{thcopositive} with  $(r,k,m_1,m_2, m_3) =(2j+1,k,2j-1,2j,2j)$, $j\in\N$.

\begin{proof}
Let $n\ge r-1$ and   $A>0$ be fixed, and  define $f\in C^r$ as
\[
f(x) =     \frac{1}{(r-2)!} \int_{-\e}^x (x-t)^{r-2} f^{(r-1)}(t) dt ,
\]
where
\[
f^{(r-1)}(x) :=
\begin{cases}
0, & \text{if $x<-\e$ or $x>0$,} \\
 x^2 (x+\e)^2/\e^3  , & \text{if $-\e \le x \le 0 $,}
\end{cases}
\]
and $\e\in (0,1/2)$.
Then,  $\norm{f^{(r)}}{} \le 1$, $f\equiv 0$ on $[-1, - \e]$,    and   $f$ is clearly nonnegative on $I$.
Now, for $x\ge 0$, we have
\be \label{apr29.1}
f(x) = Q(x)   :=   \frac{1}{(r-2)!  } \cdot \frac{1}{\e^3}  \int_{- \e}^0 (x-t)^{r-2} t^2(t+\e)^2   dt  .
\ee
Note that $Q\in\Poly_{r-2}$,
\be \label{march7}
\norm{Q}{} \le \frac{2^{r-2}}{(r-2)!  } \cdot \frac{1}{\e^3}  \int_{- \e}^0   t^2(t+\e)^2   dt \le \frac{2^{r-2}}{(r-2)!  } \e^2 =:
 c_0 \e^2 ,
\ee
and
\[
Q(-\e) = (-1)^r  {2(r^2-r) \over (r+3)!} \e^r =: -c_1 \e^r , \quad c_1 = c_1(r) >0,
\]
since $r$ is odd.

Suppose that there exists a polynomial $P_n \in \Pn$ which is nonnegative on $[-1/2,1/2]$, satisfies  $P_n^{(i)}(0)=f^{(i)}(0)=Q^{(i)}(0)$, $0\le i \le r-1$,   and such that
\be \label{tmptr22}
\norm{f-P_n}{} \le  A \norm{f^{(r)}}{} \le  A .
\ee
Then, $P_n$ can be written as
\[
P_n (x) = Q (x)
  + a_{r} x^{r}  + a_{r+1} x^{r+1} +   \dots + a_n x^n ,
\]
and taking into account that
\[
\norm{P_n-Q}{[0,1]} = \norm{P_n-f}{[0,1]} \le   A   ,
\]
by Markov's inequality, we conclude that
\[
|a_i| =  |P_n^{(i)}(0)|/i!  =   |P_n^{(i)}(0)-Q^{(i)}(0)|/i! \le K  , \quad r\le i \le n,
\]
where the constant  $K$ depends only on $A$, $r$  and $n$.

Since $r+1$ is even, for any $x\in I$, we have the following estimate
\begin{align*}
0\le P_n(x) & \le Q(x) + a_r x^r + |x|^{r+1} \sum_{i=r+1}^n |a_i| |x|^{i-r-1} \\
%& \le Q(x) + a_r x^r + x^{r+1} (nK) \\
& \le   Q(x) + a_r x^r + B x^{r+1} ,
\end{align*}
where  $B:=nK$. % is a positive constant that depends only on $A$, $r$  and $n$, and is independent of $\e$.
We also note that $a_r$ satisfies $|a_r| \le K < B$.

Now, let $x_0= -a_r/(2B)$. Then, $x_0\in (-1/2,1/2)$ and
\[
0\le  P_n(x_0) \le  Q(x_0) +(-1)^r  \frac{a_r^{r+1}}{2^{r+1} B^r}     =            Q(x_0) - \frac{a_r^{r+1}}{2^{r+1} B^r} ,
\]
and so by \ineq{march7}  (recall that $r+1$ is even)
\[
a_r^{r+1} \le 2^{r+1} B^r Q(x_0) \le c_0 2^{r+1} B^r \e^2 \quad \Rightarrow \quad |a_r| \le D \e^{2/(r+1)} ,
\]
where $D$ is a positive constant that depends only on $A$, $r$  and $n$, and is independent of $\e$.

Now, we have
\begin{align*}
0\le P_n(-\e) & \le Q(-\e) +  D \e^{2/(r+1)} \e^r  + B \e^{r+1}  \\
& = \e^r \left( - c_1 + O(\e^{2/(r+1)}) \right) < 0 \quad \text{as $\e\to 0^+$} ,
\end{align*}
which is a contradiction.
\end{proof}

The following lemma closes a few gaps left in \lem{lemma50}. It is applicable in the case $r=1$   as well as for establishing negative results for interpolation at the endpoints of $I$ (the case for $\lambda=1$).
However, if $r\ge 3$ and $\lambda \ne 1$, then this lemma is   weaker than \lem{lemma50}.

\begin{lemma}[copositive approximation with interpolatory constraints: negative result for $f\in C^r$ with odd $r\ge 1$]\leavevmode \label{lemma28}
If $r\in\N$ is \underline{odd}, then for any $n\ge r$,  $A>0$  and $\lambda\in [0,1]$, there exists $f\in C^r$ which is
identically $0$ on $[-1,-1/2]$, is nonnegative on $I$ and such that, for any nonnegative on $(\lambda-1/n,\lambda)$  polynomial $P_n \in \Pn$ satisfying
$P_n^{(i)}(\lambda)=f^{(i)}(\lambda)$, $0\le i \le r$,   we have
\[
\norm{f-P_n}{} > A \norm{f^{(r)}}{} .
\]
\end{lemma}

\begin{proof} Let $n\ge r$ and $A \ge 1$ be fixed, and define $f\in C^r$ as
\[
f(x) =     \frac{1}{(r-1)!} \int_{\lambda-\e}^x (x-t)^{r-1} f^{(r)}(t) dt ,
 \]
 where
 \[
f^{(r)}(x) :=
\begin{cases}
0 , & \text{if $x<\lambda-\e$,} \\
  (x-\lambda+\e)/\e   , & \text{if $\lambda-\e \le x \le \lambda $,} \\
 1 , & \text{if $x> \lambda$,}
\end{cases}
\]
and $\e\in(0,1/(2n))$. Note that $f$ is nonnegative on $I$,  $f\equiv 0$ on $[-1,\lambda-\e]$,  $\norm{f^{(r)}}{} = 1$ and, for $x\ge \lambda$,
\begin{align*}
f(x) =Q (x) & :=
\frac{1}{(r+1)!} \cdot \frac 1\e \left[ (x-\lambda+\e)^{r+1}-(x-\lambda)^{r+1} \right] .
\end{align*}

Now, suppose that $P_n \in\Pn$ is such that $P_n\ge 0$ on $(\lambda-1/n,\lambda)$,  $P_n^{(i)} (\lambda) = f^{(i)}(\lambda)=Q^{(i)}(\lambda)$, $0\le i \le r$, and
$\norm{f-P_n}{} \le  A \norm{f^{(r)}}{} =A$.
Then,
\[
\norm{P_n}{[-1,\lambda]} \le \norm{P_n-f}{[-1,\lambda]} + \norm{f}{[-1,\lambda]} \le A\norm{f^{(r)}}{}+ \norm{f}{[\lambda-\e,\lambda]} \le (A+\e^r/r!) \norm{f^{(r)}}{} \le 2A .
\]
Now, by Markov's inequality,
\[
\norm{P_n^{(r+1)}}{[-1, \lambda]} \le \left(\frac{2n^2}{\lambda+1}\right)^{r+1}   \norm{P_n}{[-1, \lambda]} \le 2 (2n^2)^{r+1} A,
\]
and so, since $Q^{(r+1)}\equiv 0$, % $Q\in\Poly_{r-1}$,
\begin{align*}
\frac{1}{(r+1)!} \e^r    &=
-Q(\lambda-\e) \le  P_n(\lambda-\e)  -Q(\lambda-\e) \\
&=
    \frac{1}{r!} \int_\lambda^{\lambda-\e} (\lambda-\e-t)^{r} \left( P_n^{(r+1)}(t)  -Q^{(r+1)}(t) \right)   dt  \\
& \le \frac{1}{(r+1)!} \e^{r+1} \norm{P_n^{(r+1)}}{[\lambda-\e, \lambda]}
\le \frac{2(2n^2)^{r+1} A}{(r+1)!} \e^{r+1} ,
\end{align*}
and we get a contradiction by letting $\e\to 0^+$.
\end{proof}

\subsection{Negative results for copositive approximation with interpolation at the points in $Y_s$} \label{section3.3}

\begin{lemma}[copositive approximation with interpolation at  $Y_s$: negative result for $f\in C^1$ and $Y_s = \{0\}$]\leavevmode  \label{lemma40}
For
 any $n\in\N$ and $A>0$, there exists $f\in C^1 \cap  \Delta^{(0)}(\{0\})$ such that, for any polynomial $P_n \in \Pn$ such that $P_n \ge 0$
  on $(0,1/n)$ and satisfies $P_n^{(i)}(0)=f^{(i)}(0)$, $i=0,1$, we have
\[
\norm{f-P_n}{} > A \w_3(f', 1) .
\]
\end{lemma}

\begin{proof}
Let $n\in \N$ and $A>0$ be fixed (without loss of generality, we can assume that $n\ge 3$), and  define $f\in C^1$ as
\[
f(x) =       \int_{0}^x   f'(t) dt ,
\quad
\text{ where }\;
f'(x) :=
\begin{cases}
x  (x- \e)/\e^2 , & \text{if $x<0$ or $x>\e$,} \\
0, & \text{if $0\le x \le \e $,}
\end{cases}
\]
and $\e\in (0,1/(2n))$. Note that $f\in \Delta^{(0)}(\{0\})$, $f^{(i)}(0)=0$, $i=0,1$,
and
denoting $q(x):=  x  (x-\e)/\e^2$, we have
\[
\w_3(f', 1) = \w_3(f'-q,1) \le 8 \norm{f'-q}{} = 8 \norm{q}{[0,\e]} = 2 .
\]
Then, for $x\le 0$, we have
\[
f(x) = Q(x)   :=   (2x^3-3\e x^2)/(6\e^2) .
\]
%Note that $Q(\e)= -\e/6$.

Suppose that $P_n \in \Pn$ is nonnegative on $(0,1/n)$, satisfies  $P_n^{(i)}(0)=f^{(i)}(0)=Q^{(i)}(0)=0$, $i=0,1$,   and such that
\be \label{tmptr2}
\norm{f-P_n}{} \le  A \w_3(f', 1) \le 2A .
\ee
 Taking into account that
\[
\norm{P_n-Q}{[-1,\e]} \le  \norm{P_n-f}{[-1,\e]} + \norm{f-Q}{[-1,\e]}  \le    \norm{P_n-f}{} + \norm{f-Q}{[0,\e]} \le 2 A +1   ,
\]
by Markov's inequality, we conclude that
\[
\norm{P_n'' -Q'' }{[-1,\e]} \le \left(\frac{2n^2}{\e+1}\right)^{2} \norm{P_n-Q}{[-1,\e]}  \le 4n^4(2A+1) .
\]

%\bc
%Easier way to get a contradiction: $P_n''(0)=0$ and $Q''(0) = -1/\e$
%\ec
%
%

Now,
\begin{align*}
\e/6   &= -Q(\e) \le  P_n(\e)  -Q(\e)  = \int_0^\e (\e-t) \left(P_n''(t)-Q''(t)\right) dt
\le 2n^4(2A+1)  \e^{2} ,
\end{align*}
and we get a contradiction by letting $\e \to 0^+$.
\end{proof}

\begin{lemma}[copositive approximation with interpolation at  $Y_s$: negative result for $f\in C^2$ and $Y_s = \{0\}$]\leavevmode  \label{lemmam22}
For
 any $n\in\N$ and $A>0$, there exists $f\in C^2 \cap  \Delta^{(0)}(\{0\})$ such that, for any polynomial $P_n \in \Pn\cap  \Delta^{(0)}(\{0\})$ such that   $P_n^{(i)}(0)=f^{(i)}(0)$, $i=0,1$, we have
\[
\norm{f-P_n}{} > A \w_4(f'', 1) .
\]
\end{lemma}

\begin{proof}
Let $n\in \N$ and $A>0$ be fixed (without loss of generality, we can assume that $n\ge 5$), and  define $f\in C^2$ as
\[
f(x) =       \int_{0}^x (x-t)   f''(t) dt ,
\quad
\text{ where }\;
f''(x) :=
\begin{cases}
x  (x^2- \e^2)/\e^3 , & \text{if  $|x|>\e$,} \\
0, & \text{if $|x| \le \e $,}
\end{cases}
\]
and $\e\in (0,1/2)$. Note that $f\in \Delta^{(0)}(\{0\})$, $f^{(i)}(0)=0$, $i=0,1$,
and
denoting $q(x):=  x  (x^2-\e^2)/\e^3$, we have
\[
\w_4(f'', 1) = \w_4(f''-q,1) \le 16 \norm{f''-q}{} = 16 \norm{q}{[-\e,\e]} \le 7.
\]
Then, for $x\ge \e$, we have
\[
f(x) = Q(x)   :=  \int_{\e}^x (x-t) q(t)   dt .
%  \frac{1}{\e^3} \int_{\e}^x (x-t)(t^3-\e^2 t)  dt    =    \frac{1}{60 \e^3}  (x-\e)^3 (3x^2+9x\e + 8 \e^2)   .
\]
%Note that $Q(\e)= -\e/6$.

Suppose that $P_n \in \Pn\cap  \Delta^{(0)}(\{0\})$  satisfies  $P_n^{(i)}(0)=f^{(i)}(0)=0$, $i=0,1$,   and such that
\[
\norm{f-P_n}{} \le  A \w_4(f'', 1) \le 7A .
\]
 Taking into account that
\[
\norm{P_n-Q}{[0,1]} \le  \norm{P_n-f}{[0,1]} + \norm{f-Q}{[0,1]}  \le    \norm{P_n-f}{} + \norm{Q}{[0,\e]} \le 7 A +1   ,
\]
by Markov's inequality, we conclude that
\[
\norm{P_n''' -Q''' }{[0,1]} \le \left( 2n^2 \right)^{3} \norm{P_n-Q}{[0,1]}  \le 8n^6(7A+1) .
\]
Since $P_n \in  \Delta^{(0)}(\{0\})$  and  $P_n (0)=P_n'(0)=0$,  we have $P_n''(0) = 0$ and $P_n'''(0) \ge 0$. Hence,
\begin{align*}
1/\e     &= -q'(0) = -Q'''(0) \le  P_n'''(0)  -Q'''(0)  \le 8n^6(7A+1) ,
\end{align*}
and we get a contradiction by letting $\e \to 0^+$.
\end{proof}

\begin{lemma}[copositive approximation with interpolation at  $Y_s$: negative result for $f\in C^2$ and $Y_s = \{0\}$]\leavevmode  \label{march21}
For
 any $n\ge 2$ and $A>0$, there exists $f\in C^2 \cap  \Delta^{(0)}(\{0\})$ such that, for any polynomial $P_n \in \Pn\cap \Delta^{(0)}(\{0\})$ such that
 $P_n^{(i)}(0)=f^{(i)}(0)$, $0\le i \le 2 $, we have
\[
\norm{f-P_n}{} > A \norm{f''}{} .
\]
\end{lemma}

\begin{proof}
Let $n\ge 2$ and $A\ge 1$ be fixed, and  define $f\in C^2$ as
\[
f(x) =       \int_{0}^x \int_\e^t   f''(u) du dt ,
\quad
\text{ where }\;
f''(x) :=
\begin{cases}
-1 , & \text{if $x<0$,} \\
0, & \text{if $x >\e $,} \\
x/\e -1 , & \text{if $0 \le x \le \e $,}
\end{cases}
\]
and $\e\in (0,1/2)$. Note that $f\in \Delta^{(0)}(\{0\})$, $\norm{f''}{}=1$  and, for $x\le 0$,
\[
f(x) = Q(x) := - \frac{1}{2} x^2 + \frac{1}{2} \e x .
\]
Suppose that $P_n \in \Pn\cap \Delta^{(0)}(\{0\})$  satisfies  $P_n^{(i)}(0)=f^{(i)}(0)=Q^{(i)}(0)$, $0\le i \le 2$,   and such that
\[
\norm{f-P_n}{} \le  A \norm{f''}{}  =  A .
\]
Then,
\[
P_n(x) = Q(x) + a_3 x^3 + \dots + a_n x^n,
\]
and since
\[
\norm{P_n-Q}{[-1,0]} = \norm{P_n-f}{[-1,0]}   \le   A  ,
\]
by Markov's inequality, we conclude that
\[
 |a_i| =  |P_n^{(i)}(0)-Q^{(i)}(0)|/i! \le K  , \quad 3\le i \le n,
\]
where the constant  $K$ depends only on $A$  and $n$. Hence,
\[
0 \le P_n(2\e) \le Q(2\e) + nK (2\e)^3 = \e^2 (-1 + 8nK \e) <0 \quad \text{as $\e \to 0^+$,}
\]
which is a contradiction.
\end{proof}

\begin{lemma}[copositive approximation with interpolation at  $Y_s$: negative result for $f\in C^r$ with odd $r\ge 3$ and $Y_s = \{0\}$]\leavevmode \label{lemma30ys}
If $r\ge 3$ is \underline{odd}, then for any $n\ge r$  and $A>0$, there exists
$f\in C^r \cap  \Delta^{(0)}(\{0\})$ such that, for any polynomial $P_n \in \Pn$ such that $P_n \le 0$  on $(-1/n,0)$ and satisfies $P_n^{(i)}(0)=f^{(i)}(0)$, $0\le i \le r$,  we have
\be \label{march18}
\norm{f-P_n}{} > A \norm{f^{(r)}}{} .
\ee
\end{lemma}

 \begin{proof}
Let  $n\ge r$ and $A\ge 1$ be fixed, and let $f$ and $Q$ be the same as in the proof of \lem{lemma30} with $\lambda=0$ (but $r$ is odd now).
Then, for $x\ge 0$ (see \ineq{apr29}),
\begin{align*}
f(x) & = Q(x)  :=  \frac{1}{(r-1)!} \cdot \frac{1}{\e^2} \int_{-\e}^0  (x-t)^{r-1} (-t)(t+\e)   dt .
\end{align*}
 Note that $Q(0) = r \e^r/(r+2)!$ and
\begin{align*}
  Q(-2\e) & = \frac{1}{(r-1)!} \cdot \frac{1}{\e^2} \int_{0}^\e    (2\e-t)^{r-1}   t (\e-t)   dt
  = \frac{1}{(r+2)!} \left(2^{r+1}(r-2)+r+4 \right) \e^r .
\end{align*}
 Now,
 define
\[
\tf (x) := f(x)- Q(0) \andd \tQ(x):= Q(x)-Q(0) .
\]
Then,  $\tf\in C^r \cap  \Delta^{(0)}(\{0\})$, and we now
 suppose that $P_n \in\Pn$ is such that $P_n\le 0$ on $(-1/n,0)$,  $P_n^{(i)} (0) = \tf^{(i)}(0)=\tQ^{(i)}(0)$, $0\le i \le r$, and
 $\norm{\tf-P_n}{} \le  A \norm{\tf^{(r)}}{} =A/4$.
  As in the proof of \lem{lemma30},
 \[
\norm{P_n}{[-1,0]} \le \norm{P_n-\tf}{[-1,0]} + \norm{\tf}{[-1,0]} \le A\norm{\tf^{(r)}}{}+ \norm{f}{[-\e,0]}  + |Q(0)|    \le A ,
\]
and,
by Markov's inequality,
\[
\norm{P_n^{(r+1)}}{[-1, 0]} \le \left( 2n^2 \right)^{r+1}   \norm{P_n}{[-1, 0]} \le (2n^2)^{r+1} A.
\]

 Finally,  since $\tQ^{(r+1)}\equiv 0$ and  $P_n(-2\e) \le 0$,
\begin{align*}
\frac{4}{(r+2)!} \e^r   & \le
Q(-2\e)-Q(0) =  \tQ(-2\e) \le \tQ(-2\e) -  P_n(-2\e)   \\
&=
    \frac{1}{r!} \int_0^{-2\e} (-2\e-t)^{r} \left( \tQ^{(r+1)}(t) - P_n^{(r+1)}(t)   \right)   dt  \\
& \le \frac{1}{(r+1)!} (2\e)^{r+1} \norm{P_n^{(r+1)}}{[-2\e, 0]}
\le \frac{(4n^2)^{r+1} A}{(r+1)!} \e^{r+1} ,
\end{align*}
and we get a contradiction by letting $\e\to 0^+$.
\end{proof}

\begin{lemma}[copositive approximation with interpolation at  $Y_s$: negative result for $f\in C^r$ with even $r\ge 4$ and $Y_s = \{0\}$] \leavevmode \label{lemma50ys}
If $r\ge 4$ is \underline{even}, then for
 any $n\ge r-1$ and $A>0$, there exists
 $f\in C^r \cap  \Delta^{(0)}(\{0\})$ such that, for any polynomial $P_n \in \Pn\cap \Delta^{(0)}(\{0\})$   satisfying $P_n^{(i)}(0)=f^{(i)}(0)$, $0\le i \le r-1$,  we have
\[
\norm{f-P_n}{} > A \norm{f^{(r)}}{} .
\]
\end{lemma}

\begin{proof} Let  $n\ge r$ and $A\ge 1$ be fixed, and let $f$ and $Q$ be the same as in the proof of \lem{lemma50} (but $r$ is even now). Recall (see \ineq{apr29.1})  that, for $x\ge 0$,
\[
f(x) =  Q(x) :=   \frac{1}{(r-2)!  } \cdot \frac{1}{\e^3}  \int_{- \e}^0 (x-t)^{r-2} t^2(t+\e)^2   dt
\]
and, in particular,
\[
f(0) = Q(0) =    {2 (r^2-r) \over (r+3)!} \e^r  = c_1 \e^r .
\]
Define
\[
\tf (x) := f(x)- Q(0) \andd \tQ(x):= Q(x)-Q(0) .
\]
Then, $\tf\in C^r \cap  \Delta^{(0)}(\{0\})$, $\tQ \in\Poly_{r-2}$,  and \ineq{march7} implies that
\[
\norm{\tQ}{}  \le  \norm{Q}{}  + |Q(0)| \le   c_0  \e^2  + c_1 \e^r \le (c_0+c_1)  \e^2 .
\]
We now
 suppose that $P_n \in\Pn \cap \Delta^{(0)}(\{0\})$ is such that    $P_n^{(i)} (0) = \tf^{(i)}(0)=\tQ^{(i)}(0)$, $0\le i \le r-1$, and
 $\norm{\tf-P_n}{} \le  A \norm{\tf^{(r)}}{} \le A$.
As in the proof of \lem{lemma50}, we conclude that
  $P_n$ can be written as
\[
P_n (x) = \tQ (x)
  + a_{r} x^{r}  + a_{r+1} x^{r+1} +   \dots + a_n x^n ,
\]
where
\[
|a_i|  \le K  , \quad r\le i \le n,
\]
where the constant  $K$ depends only on $A$, $r$  and $n$. This implies that
\be \label{trmarch19}
|P_n(x)-\tQ(x)-a_r x^r| \le B |x|^{r+1}, \quad x\in I,
\ee
where $B = nK$. We now have
\begin{align*}
\tQ(-3\e) & =  Q(-3\e) - Q(0) =  \frac{1}{(r-2)!  } \cdot \frac{1}{\e^3}  \int_{- \e}^0 \left[ (3\e+t)^{r-2} -t^{r-2}    \right]      t^2(t+\e)^2   dt  \\
& \ge
 \frac{2^{r-2}-1}{(r-2)!  } \cdot  \e^{r-5}  \int_{- \e}^0    t^2(t+\e)^2   dt = \frac{2^{r-2}-1}{30 (r-2)!  } \cdot  \e^{r}  =: c_2 \e^r ,
\end{align*}
and, since $P_n(-3\e) \le 0$,
\be \label{tmpmar19}
 0   \le - P_n(-3\e) \le -\tQ(-3\e)-a_r (-3\e)^r + B (3\e)^{r+1} \le \e^r \left(  -c_2  - 3^r a_r  + 3^{r+1} B \e \right).
\ee
If $a_r \ge 0$, then we immediately get a contradiction by letting $\e \to 0^+$.

Suppose now that $a_r <0$. Then,  $x_0= -a_r/(2B)$ is positive,  $P_n(x_0) \ge 0$, and so \ineq{trmarch19} implies
\[
0 \le P_n(x_0) \le \tQ(x_0) + a_r x_0^r + B x_0^{r+1} = \tQ(x_0) +(-1)^r  \frac{a_r^{r+1}}{2^{r+1} B^r} = \tQ(x_0) +   \frac{a_r^{r+1}}{2^{r+1} B^r}.
\]
Hence,
\[
-a_r^{r+1} \le 2^{r+1}B \tQ(x_0) \le 2^{r+1}B (c_0+c_1)  \e^2
\]
and, since $r+1$ is odd and $a_r <0$, we have
\[
|a_r| = -a_r \le \left[ 2^{r+1}B (c_0+c_1) \e^2 \right]^{1/(r+1)} .
\]
Therefore, \ineq{tmpmar19} implies
\[
0 \le \e^r \left( -c_2 +   O(\e^{2/(r+1)}) \right)  <0 \quad \text{as $\e \to 0^+$,}
\]
and we again arrive at a contradiction.
\end{proof}

\sect{Spline approximation with interpolatory constraints}

Given a collection of $\s$ points $Y=\{\y_j\}_{j=0}^{\s-1}$  with possible repetitions, $\y_0\le \y_1 \le \dots \le \y_{\s-1}$, we recall that, for each $j$, the multiplicity $m_j$ of $\y_j$ is the number of $\y_i$ such that $\y_i=\y_j$. Also, we let $l_j$ be the number of $\y_i=\y_j$ with $i\le j$. Suppose that a function $f$ is defined at all points in $Y$ and, moreover, for each $\y_j \in Y$, $f^{(l_j-1)}(\y_j)$ is defined as well. In other words, $f$ has $m_j-1$ derivatives at each point that has multiplicity $m_j$. Then, there is a unique Lagrange-Hermite polynomial $L_{\s-1}(\cdot; f, Y)$ of degree $\le \s-1$ that satisfies
\[ % \label{hermite}
L_{\s-1}^{(l_j-1)}(\y_j; f, Y) = f^{(l_j-1)}(\y_j), \quad \text{for all }\; 0\le j \le \s-1 .
\]

The following lemma  is a generalization of the classical Whitney's theorem. It is used for construction of local polynomial pieces near interpolation points in Lemmas~\ref{omega3} and \ref{omega2}.

\begin{lemma}[see  \cite{kls-umzh}*{Theorem 5.2 and Lemma 3.1} or \cite{kls21}*{Lemma 1.12}]\leavevmode \label{umzhlem}
Let $r\in \N_0$ and $\s\in\N$ be such that $\s\ge r+1$,   and suppose that a set $V=\{\y_j\}_{j=0}^{\s-1} \subset [a,b]$ is such that, if  $\s\ge r+2$, then
$\Lambda_r(V) \ge \lambda (b-a)$,
where $0<\lambda \le 1$  and
\[
%\Lambda(V) :=
\Lambda_r(V) :=   \min_{0\le j \le \s-r-2}  (\y_{j+r+1}-\y_j) , \quad  \; \s\ge r+2 .
\]
(If $\s=r+1$, this condition is not needed.)  If $f\in C^r[a,b]$  then, for all $x\in [a,b]$,
\[
|f(x)-L_{\s-1}(x; f, V)| \le c(\s,\lambda) (b-a)^r \w_{\s-r}(f^{(r)}, b-a; [a,b]) .
\]
\end{lemma}

\subsection{Local construction}

\begin{lemma}[onesided and intertwining approximation with interpolatory constraints: $f\in C^r$, $r\ge 1$] \leavevmode \label{lem32gen}
Suppose that $r\in\N$,  $f\in C^r[a,b]$ and $\lambda \in [a,b]$. Then, there exists $p := p(\cdot; f, k, r, \lambda)  \in\Poly_{k+r-1}$ such that  $p^{(i)}(\lambda)=f^{(i)}(\lambda)$, $0\le i\le r-1$,
\be \label{in2m9}
\norm{f-p}{[a,b]} \le c (b-a)^r \w_k(f^{(r)}, b-a; [a,b]), \quad \text{where $c=c(k,r)$,}
\ee
and
\begin{enumerate}[\rm (i)]
\item if $\lambda \in (a,b)$, then $p - f \ge 0$ on $[\lambda, b]$,  and  $(-1)^r\left( p - f \right) \ge 0$ on $[a,\lambda]$,
\item if $\lambda=a$, then $p - f \ge 0$ on $[a, b]$,
\item if $\lambda=b$, then $(-1)^r \left(p - f\right)  \ge 0$ on $[a, b]$.
\end{enumerate}
\end{lemma}

\begin{proof}
Since $f^{(r)}\in C$, there exists a polynomial $q \in \Poly_{k-1}$ such that
\[
\norm{f^{(r)}-q}{[a,b]} \le c \w_k(f^{(r)}, b-a; [a,b]) \andd q(x) \ge f^{(r)}(x), \; x\in [a,b] .
\]
We now define
\[
p(x) := \frac{1}{(r-1)!}  \int_\lambda^x (x-t)^{r-1} q(t) dt + \sum_{i=0}^{r-1} \frac{1}{i!}    f^{(i)}(\lambda)(x-\lambda)^{i}
\]
and note that $p^{(i)}(\lambda)=f^{(i)}(\lambda)$, $0\le i \le r-1$, and
\[
p(x)-f(x) = \frac{1}{(r-1)!} \int_\lambda^x (x-t)^{r-1} \left( q(t)-f^{(r)}(t) \right)  dt, \quad x\in [a,b].
\]
This  immediately implies   \ineq{in2m9}, and all other properties of $p$ are easily verified.
\end{proof}

\begin{lemma}[positive approximation with interpolatory constraints: $f\in C$]~ \label{lem33}
Suppose that $f\in C [a,b]$ is nonnegative on $[a,b]$, and $\lambda \in [a,b]$ is such that either $\lambda$ is one of the endpoints of $[a,b]$, or $\min\{|a-\lambda|, |b-\lambda|\} \ge \gamma  |b-a|$, $\gamma>0$. Then, there exists $l  \in\Poly_{1}$ such that $l$ is nonnegative on $[a,b]$, $l(\lambda)=f(\lambda)$ and
\be \label{omega2in}
\norm{f-l}{[a,b]} \le c  \w_2(f, b-a),
\ee
where $c$ is an absolute constant if $\lambda$ is an endpoint of $[a,b]$, or $c=c(\gamma)$ if $\lambda \in (a,b)$.
\end{lemma}

\begin{proof}
Let $l_1, l_2\in\Poly_1$ be  linear polynomials interpolating $f$ at $a$ and $\lambda$, and at $\lambda$ and $b$, respectively. Then, by Whitney's inequality,  $\norm{f-l_i}{[a,b]} \le c  \w_2(f, b-a)$, $i=1,2$, and if at least one of them is nonnegative on $[a,b]$, then we let $l$ to be that nonnegative linear polynomial (this is obviously the case if $\lambda$ is an endpoint of $[a,b]$). Otherwise, $l_1(b) <0$ and $l_2(a) <0$, and so
\[
0\le f(a) <  f(a)-l_2(a) \le c \w_2(f, b-a) \andd 0\le f(b) <  f(b)-l_1(b) \le c \w_2(f, b-a) .
\]
Also, since in this case $f(\lambda) < \min\{ f(a), f(b)\}$, we  have $0\le f(\lambda) \le c \w_2(f, b-a)$. This implies that $\norm{l_i}{[a,b]} \le c \w_2(f, b-a)$, $i=1,2$, and so $\norm{f}{[a,b]} \le c \w_2(f, b-a)$. Hence, we can set $l \equiv  f(\lambda)$ on  $[a,b]$.
\end{proof}

It follows from the following lemma that the constant $c$ in \ineq{omega2in} cannot be made independent of $\gamma$.  At the same time, it is rather obvious that, if $\w_2$ in the estimate \ineq{omega2in} is replaced by $\w$, then  $c$ no longer  depends on $\gamma$.

\begin{lemma}[positive approximation with interpolatory constraints: negative result for $f\in C$ and $\lambda$ near the endpoints]  \leavevmode \label{tmplem}
For any $n\in\N$ and $A>0$, there exist  $f\in C$ which is nonnegative on $I$ and $\lambda \in (-1,1)$  such that, for any nonnegative on $I$ polynomial $P_n \in \Pn$ satisfying $P_n(\lambda)=f(\lambda)$, we have
\[
\norm{f-P_n}{} > A \w_2(f, 1) .
\]
\end{lemma}

\begin{proof}
Let $n\in \N$ and $A>0$ be fixed, pick $f(x) := \max\left\{(x-\lambda)/(\lambda+1), 0 \right\}$ and suppose that $P_n \in \Pn$ is nonnegative on $I$, satisfies  $P_n(\lambda)=f(\lambda) = 0$ (this necessarily implies that $P_n'(\lambda)=0$) and such that
$\norm{f-P_n}{} \le  A  \w_2(f, 1)$.
If $Q(x) := (x-\lambda)/(\lambda+1)$, then
$\norm{f-Q}{} = \norm{Q}{[-1,\lambda]} = 1$
and
\[
\w_2(f, 1) = \w_2(f-Q, 1) \le 4 \norm{f-Q}{} = 4.
\]
Hence,
\[
\norm{P_n-Q}{} \le   \norm{P_n-f}{} + \norm{f-Q}{}   \le 4A +1 ,
\]
and so, by Markov's inequality,
\[
1/(\lambda+1) =  Q'(\lambda) =  Q'(\lambda)-P_n'(\lambda)  \le \norm{Q'-P_n'}{}  \le   n^2 \left( 4A +1 \right).
\]
We now get a contradiction by letting $\lambda \to -1^+$.
\end{proof}

\begin{lemma}[positive approximation with interpolatory constraints: $f\in C^1$] \leavevmode  \label{omega3}
Suppose that $f\in C^1[a,b]$ is nonnegative on $[a,b]$, and $\lambda \in [a,b]$  is such that $\min\{|a-\lambda|, |b-\lambda|\} \ge \gamma  |b-a|$, $\gamma >0$.  Then, there exists $p \in\Poly_{3}$ such that $p\ge 0$  on $[a,b]$, $p (\lambda)=f(\lambda)$ and
\be \label{omega3ineq}
\norm{f-p }{[a,b]} \le c (b-a) \w_3(f', b-a; [a,b]), \quad \text{where $c=c(\gamma)$.}
\ee
\end{lemma}

\begin{proof} Assume, for convenience, that $[a,b]=I$ and $\lambda = 0$.

Let $L_0(x)$ be the Lagrange-Hermite polynomial of degree $\le 3$ such that $L_0^{(i)}(0) = f^{(i)}(0)$, $i=0,1$ and $L_0(\pm 1)=f(\pm 1)$:
\[
L_0(x) = f(0)+f'(0)\, x +   \left[ \frac{f(-1)+f(1)}{2} - f(0) \right]  \, x^2 + \left[ \frac{f(1)-f(-1)}{2} - f'(0) \right]  \, x^3 .
\]
 It follows from \lem{umzhlem} with $\s=4$ and $r=1$ that
\[
\norm{f-L_0}{} \le c  \w_3(f', 1).
\]
This also implies that
\be \label{trder}
\norm{f'-L_0'}{} \le A  \w_3(f', 1),
\ee
where $A$ is some positive constant that we now consider fixed.

Consider the following two cases:

\medskip

\noindent
{\bf Case 1:} $|f'(0)| \le 2A \w_3(f',1)$ \\
We define
\[
\tp (x) := f(0)  +   \left[ \frac{f(-1)+f(1)}{2} - f(0) \right]  \, x^2 + \left[ \frac{f(1)-f(-1)}{2}   \right]  \, x^3 =: f(0) + ax^2 + bx^3
\]
and note that $\norm{L_0-\tp}{} \le |f'(0)|$ which implies that     $\norm{f-\tp}{} \le c\w_3(f',1)$. Clearly, $\tp(0)=f(0)$, and so it remains to show that $\tp$ is nonnegative on $I$.

Note that $\tp (1) = f(1)$ and $\tp (-1) = f(-1)$. If $b=0$, then $\tp$ is clearly nonnegative on $I$, and so suppose that $b\ne 0$ and, in fact,
without loss of generality,  we can assume   that $b > 0$.

Clearly,  $\tp' (x) =0 $ iff $x =0$ or $x=  -2a/(3b)=: x_0$.
Now, if $a \ge 0$, then   $x_0   \le 0$, and   $\tp\downarrow$ on $[x_0,0]$ and $\tp\uparrow$, otherwise.
 This implies that $\tp \ge 0$ on $I$.

If $a<0$, then $x_0 >0$, and $\tp\downarrow$ on $[0,x_0]$ and $\tp \uparrow$, otherwise. If $x_0\ge 1$, then $\tp$ is clearly nonnegative on $I$ .
Hence, it remains to show that $\tp(x_0) \ge 0$ in the case $0<x_0 <1$ (\ie if $0<-2a < 3 b$). Indeed, we have
\[
\tp(x_0) = f(0)+ \frac{4}{27} \cdot \frac{a^3}{b^2} > f(0)- \frac{b}{2} > \frac{f(-1)+f(1)}{2} - \frac{f(1)-f(-1)}{4} = \frac{3f(-1)+f(1)}{4} \ge  0.
\]

\medskip

\noindent
{\bf Case 2:} $|f'(0)| >  2A \w_3(f',1)$ \\
Without loss of generality, we  assume that $f'(0)<0$.
Define
\[
\tp(x):= L_0(x)+ A (x+x^2) \w_3(f',1)
\]
and note that $\norm{f-\tp}{} \le c \w_3(f',1)$ and $\tp(0)=f(0)$, and so we only need to show that $\tp$ is nonnegative on $I$.
It follows from \ineq{trder} that, for $x\in [0,1]$,
\[
\tp'(x) = L_0'(x)+ A(1+2x) \w_3(f',1) \ge L_0'(x)+ A  \w_3(f',1) \ge f'(x) ,
\]
and so
\[
\tp (x)-f(x) = \int_0^x [\tp'(u)-f'(u)]  \, du  \ge 0 , \quad 0\le x \le 1.
\]
For convenience, we denote
\[
\tp(x) = f(0)+ ax + bx^2 + cx^3 ,
\]
and note that
\[
a = \tp'(0)= L_0'(0)+ A \w_3(f',1) = f'(0)+ A \w_3(f',1) < - A \w_3(f',1) < 0 .
\]
 If $c \le 0$, then, for any $x \in [-1,0)$, we have
\[
\tp(x) =   \tp(-x) + 2a x  + 2c x^3 \ge \tp(-x) \ge 0.
\]
Hence, it   remains to consider the case   $c>0$.

If $\tp'\le 0$ on $[-1,0]$, then $\tp\downarrow$   on $[-1,0]$ and so is nonnegative there. Hence, we can assume that $\tp'$ is  positive somewhere in $[-1,0]$.
Since $c>0$, this implies that there exists $\alpha_0 \in (-1,0)$ such that $\tp' >0$ on $[-1,\alpha_0)$, and $\tp' < 0$ on $(\alpha_0,0]$. Hence, $\tp \uparrow$ on $[-1,\alpha_0]$ and $\tp\downarrow$ on $[\alpha_0, 0]$.
This implies that we only need to verify that $\tp(-1)\ge 0$ in order to conclude that $\tp$ is nonnegative on the whole interval $I$.
We have
\[
\tp(-1) = L_0(-1) = f(-1) \ge 0 ,
\]
and the proof is now complete.
\end{proof}

We remark that if $\lambda$ in the statement of \lem{omega3} is one of the endpoints of $[a,b]$, then a much stronger result is valid (see \lem{lem32gen} which implies an analogous result for positive approximation).
Also, it follows from the following lemma   that the constant $c$ in \ineq{omega3ineq} cannot be made independent of $\gamma$.

\begin{lemma}[positive approximation with interpolatory constraints: negative result for $f\in C^1$ and $\lambda$ near the endpoints]  \leavevmode \label{tmplemder}
For any $n\in\N$ and $A>0$, there exist  $f\in C^1$ which is nonnegative on $I$ and $\lambda \in (-1,1)$  such that, for any nonnegative on $I$ polynomial $P_n \in \Pn$ satisfying $P_n(\lambda)=f(\lambda)$, we have
\[
\norm{f-P_n}{} > A \w_3(f', 1) .
\]
\end{lemma}

\begin{proof}
Let $n\in \N$ and $A>0$ be fixed, pick $f(x) := \max\left\{(x-\lambda)^3/(\lambda+1)^2, 0 \right\}$ and suppose that $P_n \in \Pn$ is nonnegative on $I$, satisfies  $P_n(\lambda)=f(\lambda) = 0$ (this necessarily implies that $P_n'(\lambda)=0$) and such that
$\norm{f-P_n}{} \le  A  \w_3(f', 1)$.
If $Q(x) := (x-\lambda)^3/(\lambda+1)^2$, then
\[
\norm{f'-Q'}{} = \norm{Q'}{[-1,\lambda]} = 3 , \quad \norm{f-Q}{} = \norm{Q}{[-1,\lambda]} = \lambda+1 ,
\]
and
\[
\w_3(f', 1) = \w_3(f'-Q', 1) \le 8 \norm{f'-Q'}{} = 24.
\]
Hence,
\[
\norm{P_n-Q}{} \le   \norm{P_n-f}{} + \norm{f-Q}{}   \le 24A +2 ,
\]
and so, by Markov's inequality,
\[
\norm{P_n''-Q''}{} \le  n^4 (24A+2) =: K,
\]
and so, since $P_n^{(i)}(\lambda)= Q^{(i)}(\lambda)$, $i=0,1$,
\[
\lambda+1= -Q(-1)\le P_n(-1)-Q(-1) = \int_\lambda^{-1}(-1-t) \left[ P_n''(t)-Q''(t)\right] \, dt \le  K (\lambda+1)^2/2 ,
\]
and we get a contradiction by letting $\lambda \to -1^+$.
\end{proof}

\begin{lemma}[copositive approximation with interpolation at  $Y_s$: $f\in C^1$] \leavevmode  \label{omega2}
Suppose that $f\in C^1[a,b]$ is such that $(x-\lambda) f(x) \ge 0$, $x\in [a,b]$,
where  $\lambda \in [a,b]$  is such that $\min\{|a-\lambda|, |b-\lambda|\} \ge \gamma  |b-a|$, $\gamma >0$.  Then, there exists $p \in\Poly_{2}$ such that
$(x-\lambda) p(x) \ge 0$, $x\in [a,b]$,
  $p^{(i)}(\lambda)=f^{(i)}(\lambda)$, $i=0,1$,  and
\be \label{omega2ineq}
\norm{f-p }{[a,b]} \le c (b-a) \w_2(f', b-a; [a,b]), \quad \text{where $c=c(\gamma)$.}
\ee
\end{lemma}

\begin{proof} Assume, for convenience, that $[a,b]=I$ and $\lambda = 0$.  Then, $f\in\Delta^{(0)}(\{0\})$, and we need to construct $p\in\Delta^{(0)}(\{0\})\cap\Poly_{2}$ such that $p^{(i)}(0)=f^{(i)}(0)$, $i=0,1$, and
$\norm{f-p}{} \le c \w_2(f',1)$.

Let $L_1$ and $L_2$ be the quadratic Lagrange-Hermite polynomials   such that $L_1^{(i)}(0) = L_2^{(i)}(0) = f^{(i)}(0)$, $i=0,1$,    $L_1(-1)=f(-1)$ and $L_2(1)=f(1)$. Note that $f(0)=0$ and $f'(0)\ge 0$. Hence,
\[
L_1(x) =  f'(0)\, x +   \left[  f(-1)    +f'(0)  \right]  \, x^2
\]
and
\[
L_2(x) =  f'(0)\, x +   \left[  f(1)    -f'(0)  \right]  \, x^2 .
\]
 It follows from \lem{umzhlem} with $\s=3$ and $r=1$ that
\[
\norm{f-L_j}{} \le c  \w_2(f', 1), \quad j=1,2 .
\]
Now,
note that $L_1 \in \Delta^{(0)}(\{0\})$ iff $L_1(1) \ge 0$, and  $L_2 \in \Delta^{(0)}(\{0\})$ iff $L_2(-1) \le 0$. % and that $f'(0)\ge 0$.
So, if
$L_1(1) \ge 0$, we set $p := L_1$ and, if $L_2(-1) \le 0$, then we set $p := L_2$, and the proof is complete. Suppose now that
\[
L_1(1) = 2f'(0)+f(-1)  <  0 \andd L_2(-1)= f(1)- 2f'(0) > 0 .
\]
Then,
\[
0 \le f(1) \le f(1)-L_1(1) \le c \w_2(f', 1) \andd 0 \le - f(-1) \le L_2(-1) - f(-1) \le c \w_2(f', 1).
\]
We now   set
\[
p(x) :=  f'(0)\, x ,
\]
and note that $p\in\Delta^{(0)}(\{0\})$,  $p^{(i)}(0)=f^{(i)}(0)$, $i=0,1$, and
\[
\norm{f-p}{} \le \norm{f-L_1}{} + |f(-1)| + |f'(0)| \le
  c  \w_2(f', 1),
\]
since $0\le f'(0) <   f(1)/2  \le c \w_2(f', 1)$.
\end{proof}

 \begin{lemma}[copositive approximation with interpolation at  $Y_s$: $f\in C^2$] \leavevmode  \label{march23}
Suppose that $f\in C^2[a,b]$ is such that $(x-\lambda) f(x) \ge 0$, $x\in [a,b]$,
where  $\lambda \in [a,b]$  is such that $\min\{|a-\lambda|, |b-\lambda|\} \ge \gamma  |b-a|$, $\gamma >0$.  Then, there exists $P \in\Poly_{4}$ such that
$(x-\lambda) P(x) \ge 0$, $x\in [a,b]$,
  $P^{(i)}(\lambda)=f^{(i)}(\lambda)$, $i=0,1$,  and
\be \label{ommarch23}
\norm{f-P}{[a,b]} \le c (b-a)^2 \w_3(f'', b-a; [a,b]), \quad \text{where $c=c(\gamma)$.}
\ee
\end{lemma}

\begin{proof} Assume, for convenience, that $[a,b]=I$ and $\lambda = 0$.  Then, $f\in\Delta^{(0)}(\{0\})$, and we need to construct $P\in\Poly_4\cap \Delta^{(0)}(\{0\})$ such that $P^{(i)}(0)=f^{(i)}(0)$, $i=0,1$, and
$\norm{f-P}{} \le c \w_3(f'',1)$.

Note that $f\in\Delta^{(0)}(\{0\})$ implies that $f(0)=0$ and $f'(0) \ge 0$.

Let $p\in\Poly_2$ be any quadratic polynomial such that
\[
\norm{f''-p}{} \le c \w_3(f'', 1)
\]
 (\eg $p$ is the quadratic polynomial of best approximation of $f''$ on $I$).
We now shift $p$   so that it is below/above $f''$, \ie let
\[
p_1 := p - \norm{f''-p}{} \le f''  \andd p_2 := p + \norm{f''-p}{} \ge f'' ,
\]
and suppose that $p_i(x) = ax^2+bx+c_i $, $i=1,2$.
If $f'' \equiv p$, then there is nothing to prove, so assume that $f''\not\in\Poly_2$. Then,   $0 < c_2-c_1 = 2   \norm{f''-p}{}  \le  c \w_3(f'', 1)$. Now, for $i=1,2$, we define
\[
P_i(x) : = f'(0)x + \int_0^x (x-t) p_i(t) dt = f'(0) x +   \frac{a}{12} x^4 +  \frac{b}{6} x^3 +    \frac{c_i}{2} x^2
\]
and note that  $P_1 \le f \le P_2$ on $I$. Since $f\in\Delta^{(0)}(\{0\})$, this implies that $P_2 \ge 0$ on $[0,1]$, and $P_1 \le 0$ on $[-1,0]$:
\begin{align}  \label{condit}
Q_2(x):= P_2(x)/x & =  f'(0)   +  \frac{a}{12} x^3 +  \frac{b}{6} x^2 +    \frac{c_2}{2} x      \ge 0 , \quad  x\in [0,1],  \\ \nonumber
Q_1(x) := P_1(x)/x & = f'(0)   +  \frac{a}{12} x^3 +  \frac{b}{6} x^2 +    \frac{c_1}{2} x    \ge 0  , \quad  x\in [-1,0].
\end{align}
We now denote $\e:=(c_2-c_1)/2$ and define
\[
P(x) :=
\begin{cases}
P_1(x) +\e x^3 , & \text{if $c_2> c_1 \ge   0$,}\\
P_2(x) + \e x^3 , & \text{if $c_1<c_2 \le 0$,} \\
f'(0) x +  a x^4/12  + (b/6 +\e)  x^3   , & \text{if $c_1 <0< c_2$.}
\end{cases}
\]
It is obvious that  $P^{(i)}(0)=f^{(i)}(0)$, $i=0,1$, and $\norm{f-P}{} \le c \w_3(f'', 1)$ (in the third case, one uses the fact that $c_2-c_1 = |c_1| + |c_2| \le   c \w_3(f'', 1)$),
 and it only remains to check that  $P \in \Delta^{(0)}(\{0\})$.
We  denote $Q(x):= P(x)/x$ and note that  $P \in \Delta^{(0)}(\{0\})$ iff $Q \ge 0$ on $I$.

\begin{enumerate}[(i)]
\item  $c_2 > c_1 \ge   0$.

In this case, $P \le P_1 \le f \le 0$ on $[-1,0]$, and so we need to show that
\[ % \label{mar1}
Q(x)   =  f'(0)   +  \frac{a}{12} x^3 +  \left( \frac{b}{6}+\e \right)  x^2 +    \frac{c_1}{2} x \ge 0 , \quad x\in [0,1].
\]
If $a \ge 0$, then
\begin{align*}
Q(x)  & =    Q_1(-x) + \frac{a}{6} x^3 +\e x^2  + c_1 x  \ge 0, \quad x\in [0,1],
\end{align*}
since, by \ineq{condit},  $Q_1(-x)\ge 0$, $x\in [0,1]$, and $a, \e, c_1 \ge 0$.

If $a<0$, then $Q'$ is a concave down parabola, and $Q'(0) =c_1/2 \ge 0$. Hence, either $Q'\ge 0$ ($Q\uparrow$) on $[0,1]$, or there exists $x_0\in [0,1]$ such that
$Q' \ge 0$ ($Q\uparrow$) on $[0,x_0]$, and $Q'\le 0$ ($Q\downarrow$) on $[x_0,1]$. Therefore, using \ineq{condit} again, we have
\begin{align*}
\min_{x\in [0,1]} Q (x) & = \min\{ Q (0), Q (1)\} = \min\left\{ f'(0), Q_2(1) + \e - (c_2-c_1)/2 \right\} \\
& =  \min\left\{ f'(0), Q_2(1) \right\} \ge 0.
\end{align*}

\item $c_1 < c_2 \le 0$.

 This is analogous to the previous case. Indeed,
 $P\ge P_2 \ge f \ge 0$ on $[0,1]$, and we need to verify that
\[ % \label{mar2}
Q(x)   =  f'(0)   +  \frac{a}{12} x^3 +  \left( \frac{b}{6}+\e \right)  x^2 +    \frac{c_2}{2} x \ge 0 , \quad x\in [-1,0].
\]
 If $a \le 0$, then
\begin{align*}
Q(x)  & =    Q_2(-x) + \frac{a}{6} x^3 +\e x^2  + c_2 x  \ge 0, \quad x\in [-1,0],
\end{align*}
 since, by \ineq{condit},  $Q_2(-x)\ge 0$, $x\in [-1,0]$, and $a,  c_2 \le 0$ and $\e >0$.

If $a>0$, then $Q'$ is a concave up parabola, and $Q'(0)= c_2/2 \le 0$.
Hence, either $Q'\le 0$ ($Q\downarrow$) on $[-1,0]$, or there exists $x_0\in [-1,0]$ such that
$Q' \ge 0$ ($Q\uparrow$) on $[-1,x_0]$, and $Q'\le 0$ ($Q\downarrow$) on $[x_0,0]$. Therefore, using \ineq{condit} again, we have
\begin{align*}
\min_{x\in [-1,0]} Q (x) & = \min\{ Q (0), Q (-1)\} = \min\left\{ f'(0), Q_1(-1) + \e - (c_2-c_1)/2 \right\} \\
& =  \min\left\{ f'(0), Q_1(-1) \right\} \ge 0.
\end{align*}

\item $c_1 <0< c_2$. We need to check that
\[
Q(x) = f'(0)   +  \frac{a}{12} x^3 +  \left( \frac{b}{6}+\e \right)  x^2   \ge 0 , \quad x\in I.
\]
Denote
\[
\tQ(x) := f'(0)   -  \frac{|a|}{12} x^3 +  \left( \frac{b}{6}+\e \right)  x^2 ,
\]
and note that $Q\ge 0$ on $I$ iff $\tQ \ge 0$ on $[0,1]$.

Now, since $\tQ'$ is a concave down parabola (or a linear function) with $\tQ'(0)=0$, either $\tQ'$ does not change its sign (and so $\tQ$ is either $\uparrow$ or $\downarrow$) on $[0,1]$, or there exists $x_0\in [0,1]$, such that
$\tQ' \ge 0$ ($\tQ\uparrow$) on $[0, x_0]$ and $\tQ' \le 0$ ($\tQ\downarrow$) on $[x_0, 1]$. This implies that $\tQ$ achieves its minimum value on $[0,1]$ at $0$ or $1$, and
%so
%\begin{align*}
%\min_{x\in [-1,1]} Q (x) & = \min_{x\in [0,1]} \tQ (x) =  \min\{ \tQ (0), \tQ (1)\}
%\end{align*}
  since $\tQ (0) = f'(0) \ge 0$, it remains to verify that $\tQ(1) \ge 0$.

Using \ineq{condit}, if $a\ge 0$,   we have
\[
\tQ(1) = f'(0)   -  \frac{ a }{12}   +   \frac{b}{6}+\e  = Q_1(-1) +\e + \frac{c_1}{2} = Q_1(-1) + \frac{c_2}{2} \ge 0 ,
\]
and if $a<0$, then
\[
\tQ(1) = f'(0)   +  \frac{ a }{12}   +   \frac{b}{6}+\e  = Q_2(1)+ \e - \frac{c_2}{2} = Q_2(1) - \frac{c_1}{2} \ge 0.
\]
\end{enumerate}
The proof is now complete.
\end{proof}

\subsection{Global construction}

Let $\zn := \{z_j\}_{j=0}^n$, where $-1:=z_n<z_{n-1} <\dots <z_{1}<z_0:=1$,
$n\in\N$, be a given knot sequence on $I$, and set $z_i:=-1$, $i>n$, $z_i:=1$, $i<0$, and  $J_i:=[z_i,z_{i-1}]$.

Denote by $\Sigma_{k}(\zn)$ the collection of all
continuous piecewise polynomials of degree $\le k-1$ (of order $k$) on the  knot sequence $\zn$.
Functions from $\Sigma_{k}(\zn)\cap C^{k-2}$     %, \ie piecewise polynomials from $\Sigma_{k}(\zn)$ that have the highest smoothness without becoming polynomials on $I$,
are usually called splines (or splines with minimal defect) of order $k$ with knots $\zn$.

We need the following result which is usually referred to as Beatson's lemma.

\begin{lemma}[Beatson \cite{bea}*{Lemma 3.2}] \label{lem61}
Let $m \ge  2$ be an integer and $d = 2(m - 1)^2$. Let
$T = \{t_i\}_{i=-\infty}^\infty$  be a strictly increasing knot sequence with $t_0=a$  and $t_d = b$. If  $p_1, p_2\in\Poly_{m-1}$, then
  there exists a spline $S\in\Sigma_m(T)\cap C^{m-2}$ %of order $m$ with knot sequence $T$ (\ie $S \in C^{m-2}$ and $S|_{[t_{i-1}, t_{i}]}$ is a polynomial of degree $\le m-1$)
  such that
\begin{enumerate}[(i)]
\item  $S(x)$ is a number between $p_1(x)$ and $p_2(x)$  for each $x\in [a,b]$,
\item $S \equiv p_1$ on $(-\infty, a]$ and $S \equiv p_2$ on $[b, \infty)$.
\end{enumerate}
\end{lemma}

The following theorem is our  main result on intertwining   approximation with interpolatory constraints by splines.

 \begin{theorem}[intertwining spline approximation with interpolatory constraints]\leavevmode\label{splnthm1}%
Let $k,p\in\N$, $s\in\N_0$, $Y_s\in\Y_s$, $A_p\in\A_p$,  $(r,m_1,m_2, m_3)\in\intset_{\text{intertwining}}$, defined in \ineq{defintertwining}, and suppose that
 $f\in C^r$.
 Also, let $\zn$ be a
given knot sequence such that there are at least $4(k+r-1)^2$
knots in each nonempty open interval $(\beta_j,  \beta_{j-1})$, $1\le j \le s+p+1$, where $\{\beta_j\}_{j=1}^{s+p} =  Y_s\cup A_p$, $\beta_0 = 1$ and $\beta_{s+p+1}=-1$  (see \ineq{unionya}).
Then, there exists
\be \label{mar12}
S  \in   \Sigma_{k+r}(\zn)\cap C^{k+r-2} \cap \td\cap \I^{(m_1)}(f, A_p)\cap \I^{(m_2)}(f, Y_s)  \cap  \I^{(m_3)}(f, \{\pm 1\})
\ee
such that, for $0\le i \le n-1$,
 \be \label{eq:sp1}
\norm{f-S}{J_i} \le c|J_i|^r \w_{k}(f^{(r)}, |\jj_i|; \jj_i) , \quad c=c(k,r, \rho) ,
\ee
 where  $\rho:= \rho(\zn) := \max \left\{ |J_{i\pm1}|/|J_i| \st 1\le i \le n \right\}$, and $\jj_i$ is an
interval such that
$J_i\subset \jj_i\subseteq \bigl[z_{i+6(k+r-1)^2}, \,z_{i-6(k+r-1)^2} \bigr]$. Here, if $m_1=-1$, then
 the restriction $S\in \I^{(-1)}(f, A_p)$ is vacuous, \ie
 there is no requirement that $S$ interpolate $f$ at the points in $A_p$.
\end{theorem}

We recall that, if $s=0$, then $\Y_0 = \emptyset$, and so \thm{splnthm1} becomes a theorem on onesided spline approximation with interpolatory constraints.

\begin{proof}
The idea of the proof    is quite similar to that of \cite{hky}*{Theorem 3}.
 The only difference is that, near the points from $Y_s\cup A_p\cup \{\pm 1\}$, we   use  \lem{lem32gen}    to construct an appropriate polynomial piece of $S$.
 For completeness, we provide details.

Without loss of generality,  we may assume that $\pm 1 \in A_p$.
We now denote $m:= k+r$,    $d:=2(m-1)^2$, $N:=\lceil n/d \rceil$ and $\zbar_i:=z_{di}$, $i\in\Z$.  Note that $\zbar_i=1$ for $i\le0$ and $\zbar_i=-1$ for
$i\ge N$.

We   now   construct overlapping polynomial pieces of degree
$\le m-1$ on the coarser partition $\ozn :=\{\zbar_i\}_{i=0}^N$.
We say that the interval $\ibari:= [\zbar_i,\, \zbar_{i-1}]$ is ``$Y_s$-contaminated''
if $\zbar_i \le y_j < \zbar_{i-1}$ for some point $y_j\in Y_s$, and it is ``$A_p$-contaminated'' if
 $\zbar_i \le \alpha_j < \zbar_{i-1}$, $i\ge 2$, or $\zbar_1 \le \alpha_j \le  \zbar_{0}=1$,  for some point $\alpha_j\in A_p$.
By assumption in the statement of the theorem,
there exists exactly one $\beta_j$ in each of the
contaminated intervals $\ibar_{\mu_j}$, $1\le j \le s+p$, and there is
at least one non-contaminated interval between
 $\ibar_{\mu_j}$ and $\ibar_{\mu_{j+1}}$, \ie
 \be \label{eq:sp2}
  \mu_j+2 \leq \mu_{j+1}, \qquad 1\le j \le s+p-1.
 \ee
 Note that $\mu_1=1$ and $\mu_{s+p}=N$ by our assumption that $\pm 1\in A_p$. % {\bf and it is convenient to denote $\mu_0:=1$, and $\mu_{s+p+1}:=N$. ???????}

Also, if an interval $J\subset I$ does not contain any points from $Y_s$, then $J \subset [y_j, y_{j-1}]$, for some $1\le j \le s+1$, and we say that
$J$ is  ``positive'' if $j$ is odd, and it is ``negative'' if $j$ is even. Any function  $S\in\td$ satisfies $S\ge f$ on all positive intervals, and $S\le f$ on all negative intervals.

If $\mu_{j+1}> \mu_j+2$, on each of the   intervals $[\zbar_i,\, \zbar_{i-2}]$,
$i=\mu_j+2, \dots, \mu_{j+1}-1$, %(note that $\ibari$ is non-contaminated in this case),
 by Whitney's inequality there exist two polynomials
$P_i$ and $Q_i$ of degree $<m$ such that
 \[
 P_i(x) \ge f(x) \ge Q_i(x), \quad   x\in[\zbar_i,\, \zbar_{i-2}],
\]
and
\[
 \norm{P_i-Q_i}{[\zbar_i,\, \zbar_{i-2}]} \le c  \w_m(f , |\ibari|; [\zbar_i,\, \zbar_{i-2}]) \le c |\ibari|^r \w_k(f^{(r)} , |\ibari| ;  [\zbar_i,\, \zbar_{i-2}]), \quad c=c(m,\rho).
\]
 We now define $p_i$ on  $[\zbar_i,\, \zbar_{i-2}]$ as
 \[
 p_i :=
 \begin{cases}
 P_i , & \text{if $(\zbar_i,\, \zbar_{i-2})$ is positive,} \\
 Q_i , & \text{if $(\zbar_i,\, \zbar_{i-2})$ is negative.}
 \end{cases}
 \]
 Then,
 \be \label{min}
  \norm{f-p_i}{[\zbar_i,\, \zbar_{i-2}]}     \le c |\ibari|^r \w_k(f^{(r)} , |\ibari| ;  [\zbar_i,\, \zbar_{i-2}]), \quad c=c(k,r,\rho).
 \ee
If $\mu_{j+1}= \mu_j+2$ (\ie  if there is only one non-contaminated interval between
 $\ibar_{\mu_j}$ and $\ibar_{\mu_{j+1}}$), then the above construction is not needed.

We now construct local polynomial pieces near each contaminated interval $\ibar_{\mu_j}$, $1\le j \le s+p$.
Denote  $\itilde_{\mu_j} := [\zbar_{\mu_j+1}, \zbar_{\mu_j-2}]$ and use  \lem{lem32gen} with $[a,b]=\itilde_{\mu_j}$  to define $p_{\mu_j}$ on $\itilde_{\mu_j}$. We recall that $p(\cdot; f,k,r,\lambda)$ denotes the polynomialfrom the statement of \lem{lem32gen} and consider three cases.

\medskip

\noindent
{\bf Case 1:} $\ibar_{\mu_j}$ is $\{\pm 1\}$-contaminated (\ie $j=1$ or $j=s+p$).
\begin{enumerate}[(i)]
\item  If $j=1$ and $r$ is even, then $p_{\mu_{1}} = p_1 := p(\cdot ; f, k, r, 1)$.
\item  If $j=1$ and $r$ is odd, then $p_{\mu_{1}} = p_1 := -p(\cdot ; -f, k, r, 1)$.
\item If $j=s+p$ and $\ibar_{\mu_{s+p}}=\ibar_{N}$ is positive, then $p_{\mu_{s+p}} = p_N := p(\cdot ; f, k, r, -1)$.
\item If $j=s+p$ and $\ibar_{\mu_{s+p}}=\ibar_{N}$ is negative, then $p_{\mu_{s+p}} = p_N := -p(\cdot ; -f, k, r, -1)$.
 \end{enumerate}

\noindent
{\bf Case 2:} $1<j<s+p$ and  $\ibar_{\mu_j}$ is $A_p$-contaminated. (Recall that, in this case,  $r=1$ is excluded, and so  $r\ge 2$.)
\begin{enumerate}[(i)]
\item If  $r$ is even and  $\ibar_{\mu_j}$ is positive, then $p_{\mu_j} = p(\cdot ; f, k, r, \beta_j)$.
\item If   $r$ is even and  $\ibar_{\mu_j}$ is negative, then $p_{\mu_j} = -p(\cdot ; -f, k, r, \beta_j)$.
\item If   $r$ is odd and  $\ibar_{\mu_j}$ is positive, then $p_{\mu_j} = p(\cdot ; f, k+1, r-1, \beta_j)$.
\item If   $r$ is odd and  $\ibar_{\mu_j}$ is negative, then $p_{\mu_j} = -p(\cdot ; -f, k+1, r-1, \beta_j)$.
 \end{enumerate}

\noindent
{\bf Case 3:}  $\ibar_{\mu_j}$ is $Y_s$-contaminated (and so $\beta_j=y_i$, for some $1\le i\le s$).
\begin{enumerate}[(i)]
\item If  $r$ is odd and $i$ is odd, then $p_{\mu_j} = p(\cdot ; f, k, r, \beta_j)$.
\item If  $r$ is odd and  $i$ is even, then $p_{\mu_j} = -p(\cdot ; -f, k, r, \beta_j)$.
\item If  $r$ is even and  $i$ is odd, then $p_{\mu_j} = p(\cdot ; f, k+1, r-1, \beta_j)$.
\item If  $r$ is even and  $i$ is even, then $p_{\mu_j} = -p(\cdot ; -f, k+1, r-1, \beta_j)$.
 \end{enumerate}

We remark that, if $r\ge 2$ and  $p_{\mu_j} = p(\cdot ; f, k+1, r-1, \lambda)$, then  $p_{\mu_j}^{(i)}(\lambda) = f^{(i)}(\lambda)$, $0\le i \le r-2$, and
\[
\norm{f-p_{\mu_j}}{\itilde_{\mu_j}} \le c |\itilde_{\mu_j}|^{r-1} \w_{k+1}(f^{(r-1)}, |\itilde_{\mu_j}|; \itilde_{\mu_j}) \le
c |\itilde_{\mu_j}|^{r} \w_{k}(f^{(r)}, |\itilde_{\mu_j}| ;  \itilde_{\mu_j}).
\]
All overlapping polynomial pieces with all the right properties have now been constructed, and it remains to blend them together to obtain a smooth spline $S$ on the original knot sequence $\zn$ with the same properties.

 If both   $\ibari$ and $\ibar_{i+1}$ are
non-contaminated,
%and $i<N$,
then $p_i$ and $p_{i+1}$  overlap on $\ibari$, which contains $d-1$ interior knots from $\zn$.
By \lem{lem61}
there exists a spline
$S_i$ of order $m$ on $\ibari$ on these knots that connects with
$p_{i+1}$ and $p_i$
in a $C^{m-2}$ manner at $\zbar_i=z_{di}$ and $\zbar_{i-1}=z_{d(i-1)}$,
respectively. Moreover, the graph of $S_i$ lies between those of $p_{i+1}$
and $p_i$, and, hence,
$\sgn(p_{i+1}(x)-f(x)) = \sgn(p_{i}(x)-f(x)) = \sgn(S_i(x)-f(x))$, $x\in \ibari$. Additionally, it follows from \ineq{min} that
 \[
  \norm{f-S_i}{[\zbar_i,\, \zbar_{i-2}]}     \le c |\ibari|^r \w_k(f^{(r)} , |\ibari| ;  [\zbar_{i+1},\, \zbar_{i-2}]).
 \]
 The blending of overlapping polynomial pieces involving contaminated
intervals is done in exactly the same way.
The only difference is that
the spline pieces  $S_i$ thus
produced   satisfy the estimate above with a slightly larger interval
in place of $[\zbar_{i+1},\, \zbar_{i-2}]$ on the right-hand side,
($[\zbar_{i+2},\, \zbar_{i-3}]$ at worst), which will make no
difference in the rest of the proof.

We define the final spline $S$ on each $\ibari$ as follows: if
there is only one local polynomial $p_i$ over $\ibari$, set $S$ to this
polynomial; if there are two polynomials overlapping on $\ibari$,
then there must
be a blending local spline $S_i$, and so set $S$ to $S_i$. It is clear from this
construction that
$S\in   \Sigma_{k+r}(\zn)\cap C^{k+r-2} \cap \td$, and $S$ also satisfies interpolatory conditions stated in \ineq{mar12}.

Finally,  all neighboring intervals $I_i:=
[z_i,\, z_{i-1}]$ in the original partition $\zn$ are comparable in size
and each interval $\ibari=[z_{di},\, z_{d(i-1)}]$
contains no more than $d$ such intervals. Therefore, (\ref{eq:sp1}) is satisfied.
\end{proof}

 \begin{theorem}[copositive  spline approximation with interpolatory constraints]\leavevmode\label{splcopositive}%
 Let $p\in\N$, $s\in\N_0$,    $Y_s \in\Y_s$, $A_p\in\A_p$, and $(r,k,m_1,m_2, m_3)\in\copset_{\text{copositive}}$, defined in \ineq{defcopositive},
and suppose that
$f\in C^r\cap  \Delta^{(0)}(Y_s)$.
 Also, let $\zn$ be a
given knot sequence such that there are at least $4(k+r-1)^2$
knots in each nonempty open interval $(\beta_j,  \beta_{j-1})$, $1\le j \le s+p+1$, where $\{\beta_j\}_{j=1}^{s+p} =  Y_s\cup A_p$, $\beta_0 = 1$ and $\beta_{s+p+1}=-1$  (see \ineq{unionya}).
Then, there exists
\be \label{april21}
S  \in   \Sigma_{k+r}(\zn)\cap C^{k+r-2} \cap   \Delta^{(0)}(Y_s)     \cap \I^{(m_1)}(f, A_p)\cap \I^{(m_2)}(f, Y_s)  \cap  \I^{(m_3)}(f, \{\pm 1\})
\ee
such that, for $0\le i \le n-1$,
 \be \label{april211}
\norm{f-S}{J_i} \le c|J_i|^r \w_{k}(f^{(r)}, |\jj_i| ; \jj_i) , \quad c=c(k,r, \rho) ,
\ee
 where
 $\rho$ and  $\jj_i$ are the same as in the statement of \thm{splnthm1}.
 Here, if $m_1=-1$, then
 the restriction $S\in \I^{(-1)}(f, A_p)$ is vacuous, \ie
 there is no requirement that $S$ interpolate $f$ at the points in $A_p$.
\end{theorem}

\begin{proof}
First, we note that \thm{splcopositive} immediately follows from \thm{splnthm1} if
\begin{align*}
(r,k,m_1,m_2, m_3) \in & \left\{ (1,k,-1,0,0), (2,k,1,0,1) \st k\in\N \right\} \\
& \cup \left\{(r,k,m_1,m_2, m_3) \in  \intset_{\text{intertwining}} \st r\ge 3 \right\} .
\end{align*}
In the remaining cases,
its  proof is exactly the same as that of \thm{splnthm1} with the only difference that, while constructing local polynomial pieces near the points from $Y_s\cup A_p\cup \{\pm 1\}$, instead of  \lem{lem32gen}, we need to use:
\begin{itemize}
\item \lem{lem33} near $A_p$ and a (trivial) linear interpolant near $Y_s$, if $r=0$ and $k=2$,
\item \lem{omega3} near $A_p$, if $r=1$ and $k=3$,
\item  \lem{omega2} near   $Y_s$ and \lem{omega3} near  $A_p$, if $r=1$, $k=2$ and $m_2=1$,
\item \lem{march23} near   $Y_s$, if $r=2$, $k=3$ and $m_2=1$.
\end{itemize}
We omit details.
\end{proof}

\sect{Intertwining polynomial approximation of truncated power functions and splines with interpolatory constraints}

We reserve the special notation $\tn$ for the Chebyshev partition of $I$, \ie
 $\tn := \{x_j\}_{j=0}^n$,  where
$x_{j}:= x_{j,n} :=  \cos\left(j\pi/n\right)$, $0 \leq j \leq n$. Also, for $1\le j \le n$, we denote
\[
I_{j} := I_{j,n} :=  [x_{j}, x_{j-1}] , \quad   |I_j| = x_{j-1}-x_j   \andd \psi_{j}(x) := \psi_{j,n}(x) := \frac{|I_j|}{|x-x_j|+|I_j|} .
\]
Note that $|I_{j\pm 1}| < 3|I_j|$ and $\rho_n(x) < |I_j| < 5 \rho_n(x)$, $x\in I_j$, and so $\psi_{j\pm 1}(x) \sim \psi_j (x)$,   $x\in I$.

Additionally, we denote
\[
(x-\lambda)_+^m := (x-\lambda)^m \chi[\lambda, 1](x) , \quad m\ge 0,
\]
where
\[
\chi[a,\,b](x) := \left\{
\begin{array}{ll}
1 , & \mbox{\rm if }  x\in[a,\,b], \\
0,  & \mbox{\rm otherwise.}
\end{array}
 \right.
  \]
It is also convenient to denote $x_j := 1$ for $j<0$, and $x_j := -1$ for $j>n$.

\subsection{Approximation of truncated power functions}

%The same construction as in \cite{hky}*{Section 3.1} can be used to construct polynomials that are intertwining with truncated power functions, approximate them well and, in addition, also satisfy (Hermite) interpolatory constrains.

\begin{lemma}[intertwining  approximation of truncated powers with interpolatory constraints]  \leavevmode \label{newlemma}
Let $s,p,r,m \in\N_0$, $Y_s \in\Y_s$ and $A_p \in \A_p$, and
suppose that  $j,n\in\N$ are such that   the interval $(x_{j+1}, x_{j-2})$ does not contain any points from $Y_s \cup A_p \cup\{\pm 1\}$, and let
$\lambda\in [{x_{j}},{x_{j-1}}]$ and $g_{m,\lambda} (x):= (x-\lambda)_+^{m}$.

Then, there exists a sufficiently large constant $\mu = \mu(s,p,r,m)$ and
 polynomials $R_1$ and $R_2$ of degree $\le c(\mu) n$ such that
\be \label{ii1}
 R_1, R_2 \in \I^{(r)}(g_{m,\lambda}, A_p\cup Y_s \cup\{\pm 1\}) ,
 \ee
\be \label{ii2}
\left\{ R_1 - g_{m,\lambda} , g_{m,\lambda} - R_2 \right\}  \subset  \Delta^{(0)}(Y_s) ,
\ee
\ie  $\{R_1,R_2\}$ is an intertwining pair of  polynomials  for $g_{m,\lambda}$ with respect to
$Y_s$,
and
\be \label{ii3}
\left|R_i(x) - g_{m,\lambda}(x) \right| \leq c(\mu) \psi_j(x)^{\mu-m} |I_j|^{m} ,  \quad \text{for all $x\in I$ and $i=1,2$.}
\ee
\end{lemma}

\begin{proof} In order to avoid possible confusion, we note that, in this proof,  the prime notation is used to denote new parameters/functions  and is  not  used to denote any  derivatives.
It is also clear that, without loss of generality, we can assume that  $\pm 1 \in A_p$.

\medskip\noindent
{\bf Case 1: even $m$.}
For even $m$ and without \ineq{ii1}, \lem{newlemma} is \cite{hky}*{Lemma 19} with
\[
\{a_1, \dots, a_k\} =   \left\{ y  \in Y_s \st y  < x_j \right\}
\andd
\{b_1, \dots, b_l\} =   \left\{ y  \in Y_s \st y  >  x_j \right\} .
\]
For even $m$, the same construction also works if points $a_i$'s and $b_i$'s are allowed to coalesce, which implies Hermite interpolation of $g_{m,\lambda}$ by $R_i$'s at these points.
Hence, % assuming without loss of generality that $\pm 1 \in A_p$,
 it is  sufficient to set
\[
\{a_1, \dots, a_k\} :=  \bigcup_{y \in Y_s, \; y < x_j}
  \big\{ \underbrace{y , \dots, y}_{r'} \big\}
\cup
\bigcup_{\alpha \in A_p, \; \alpha < x_j} \big\{ \underbrace{\alpha , \dots, \alpha}_{r''} \big\}
\]
and
\[
\{b_1, \dots, b_l\} :=
 \bigcup_{y \in Y_s, \; y > x_j}
  \big\{ \underbrace{y , \dots, y}_{r'} \big\}
\cup
\bigcup_{\alpha \in A_p, \; \alpha > x_j} \big\{ \underbrace{\alpha , \dots, \alpha}_{r''} \big\} ,
\]
where
$r'$ and $r''$ are, respectively, odd and even integers which are not less than $r+1$ (for example, we can set $r' := 2r+1$ and $r'' := 2(r+1)$),
 so that  $R_i - g_{m,\lambda}$, $i=1,2$, changes its sign  at the points in $Y_s$ and does not change it  at the points in  $A_p$.

A byproduct of the construction in \cite{hky}*{Lemma 19} was the fact that $R_i(\lambda) = 0$, $i=1,2$ in the case $m>0$. This makes it impossible for the same construction to work for odd $m$. For example, if $m=1$, then it is clear that there does not exist a polynomial $R\in\Pn$ such that $R(\lambda)=0$ and, at the same time, $R \ge g_{1,\lambda}$ on $(\lambda-\e, \lambda+\e)$.

\medskip
\noindent
{\bf Case 2: odd  $m$.} Suppose that $m$ is odd, and  let $n':=2n$, $j' := 2j$ and
 $Y_s' := Y_s \cup \{ y' \}$, where $y' := x_{j'+1,n'}$.

 Note that
\[
x_{j', n'}  = x_j, \quad  I_{j'+1,n'} = [x_{j'+1, n'}, x_{j', n'}]\subset I_j \andd \dist\{y', \{x_j, x_{j+1}\} \ge |I_{j+1}|/4 .
\]
It now follows from the statement of the lemma for even $m$ (Case 1 above) that there exist polynomials $R_1'$ and $R_2'$ such that
\be \label{ii11}
 R_1', R_2' \in \I^{(r)}(g_{m-1,\lambda}, A_p\cup Y_s' ) ,
 \ee
\be \label{ii21}
\left\{ R_1' - g_{m-1,\lambda} , g_{m-1,\lambda} - R_2' \right\}  \subset  \Delta^{(0)}(Y_s') ,
\ee
and
\be \label{ii31}
\left|R_i'(x) - g_{m-1,\lambda}(x) \right| \leq c(\mu) \psi_j(x)^{\mu-m+1} |I_j|^{m-1} ,  \quad \text{for all $x\in I$ and $i=1,2$.}
\ee
We  define
\[
R_i''(x):= (x-\lambda) R_i'(x), \quad i=1,2 ,
\]
and note that $R_1'', R_2'' \in \I^{(r)}(g_{m,\lambda}, A_p\cup Y_s  )$ and, for all $x\in I$ and $i=1,2$,
\be \label{april24}
\left|R_i''(x) - g_{m,\lambda}(x) \right| \leq c(\mu) |x-\lambda| \psi_j(x)^{\mu-m+1} |I_j|^{m-1}  \le c(\mu)   \psi_j(x)^{\mu-m} |I_j|^{m}  ,
\ee
since
\[
|x-\lambda| \le |x-x_j| +|I_j| = \psi_j(x)^{-1}  |I_j|.
\]
Let $J$ be the index such that $y_i \in [x_{j-2},1]$  for $1\le i \le J$, and $y_i \in [-1,x_{j+1}]$  for $J+1\le i \le s$, and note that, if  $J=0$ or $J= s$, then there are no $y_i$'s in the intervals
$[x_{j-2},1]$  or $[-1,x_{j+1}]$, respectively.

Now, letting $h$ be either $R_1' - g_{m-1,\lambda}$ or  $g_{m-1,\lambda} - R_2'$, we
 note that
\[
h \in \Delta^{(0)}(Y_s')\quad \iff \quad
\begin{cases}
(-1)^i h \ge 0  &  \text{on $[y_{i+1}, y_i]$, for $0\le i \le J-1$,}   \\
(-1)^J h \ge 0  & \text{on $[y', y_J]$,}\\
(-1)^{J+1} h \ge 0  & \text{on $[y_{J+1}, y']$,}\\
(-1)^{i+1}  h \ge 0  &  \text{on $[y_{i+1}, y_i]$, for $J+1\le i \le s+1$,}
\end{cases}
\]
and so, for $h' := (\cdot - \lambda) h$, we have
\[
\begin{cases}
(-1)^i h' \ge 0  &  \text{on $[y_{i+1}, y_i]$, for $0\le i \le J-1$,}   \\
(-1)^J h' \ge 0  &  \text{on $[\lambda, y_J]$,}   \\
(-1)^{J+1} h' \ge 0  & \text{on $[y',\lambda]$,}\\
(-1)^{J} h' \ge 0  & \text{on $[y_{J+1}, y']$,}\\
(-1)^{i}  h' \ge 0  &  \text{on $[y_{i+1}, y_i]$, for $J+1\le i \le s+1$.}
\end{cases}
\]
In other words, all the right inequalities for $h' \in \Delta^{(0)}(Y_s)$ are satisfied except that  the restriction  $(-1)^J h' \ge 0$ that we need on $[y_{J+1}, y_J]$ fails on its small subinterval  $[y',\lambda]$.
Hence, the polynomials $R_i''$, $i=1,2$,  need to be corrected on $[y',\lambda]$
 in order to guarantee that
$ \widetilde R_1'' - g_{m,\lambda} , g_{m,\lambda} - \widetilde R_2'' \in  \Delta^{(0)}(Y_s)$, where $\widetilde R_i''$, $i=1,2$, are these corrected polynomials.

We note that it follows from \ineq{april24} that
\[
\norm{h'}{}  \le c(\mu) \norm{ \psi_j^{\mu-m}}{} \cdot |I_j|^{m} \le c(\mu) |I_j|^{m} := c_1 |I_j|^{m}.
\]
We now recall (see \eg \cites{dzya})  that the algebraic polynomial of degree $\le 4n-2$,
\[ t_{j}(x):=
(x-x_j^0)^{-2} \cos^2 2n\arccos x  +
(x-\xb_j)^{-2} \sin^2 2n\arccos x ,
 \]
 where
 \[
 \xb_{j}:=\cos\left(\frac{j\pi}{n}-\frac{\pi}{2n}\right) \andd
 x_{j}^0 :=
 \begin{cases}
 \cos\left(\frac{j\pi}{n}-\frac{\pi}{4n}\right) , &  \text{if $j < n/2$,}\\
 \cos\left(\frac{j\pi}{n}-\frac{3\pi}{4n}\right), &  \text{if $j \ge  n/2$,}\\
 \end{cases}
 \]
 satisfies
 \[
 |I_j|^{-2} \psi_j(x)^2 \le t_j(x) \le 4000 |I_j|^{-2} \psi_j(x)^2, \quad \text{$x\in I$ and $1\le j\le n-1$.}
 \]
 We now  define
 \[
 H(x):= t_j(x)^{\mu'}   |I_j|^{2\mu'}  \prod_{y \in Y_s} \frac{ (x-y)^{2r+1} }{|x_j-y|^{2r+1}}   \prod_{\alpha \in A_p} \frac{ (x-\alpha)^{2r+2} }{|x_j-\alpha|^{2r+2}}
 \]
 with $\mu' :=  \mu + (2r+2)(s+p)$, for example,
 and note that
 \[
 H  \in \I^{(r)}(0, A_p\cup Y_s) \cap \Delta^{(0)}(Y_s)
 \]
 (in particular, $(-1)^J H \ge 0$ on $[y_{J+1}, y_J]$),
\begin{align*}
|H(x)|  &  \le c    \psi_j(x)^{2\mu'}  \prod_{y \in Y_s} \frac{|x-y|^{2r+1} }{|x_j-y|^{2r+1}}  \prod_{\alpha \in A_p} \frac{ |x-\alpha|^{2r+2} }{|x_j-\alpha|^{2r+2}} \\
& \le
 c    \psi_j(x)^{2\mu'-(2r+2)(s+p)}
 \le
 c    \psi_j(x)^{\mu} ,
\quad x\in I,
\end{align*}
since, for $\beta  \in Y_s \cup A_p$,
\begin{align*}
 \frac{ |x-\beta|}{|x_j-\beta| } & \le 1+ \frac{ |x-x_j|  }{|x_j-\beta| } \le 1 + \frac{ |x-x_j| + |I_j|}{|I_j|} \cdot \frac{|I_j|}{|x_j-\beta| } \\
 & \le 1 + \psi_j(x)^{-1} \cdot \frac{|I_j|}{\min\{ |I_{j-1}|, |I_{j+1}|\}}
 \le 1 + 3 \psi_j(x)^{-1} \le 4 \psi_j(x)^{-1} .
\end{align*}
Also, for $x\in [y',\lambda]$,  we have
\[
t_j(x) \ge (|x-x_j|+|I_j|)^{-2} \ge \left( \max\{|y'-x_j|+|I_j|, |\lambda-x_j|+|I_j| \}\right)^{-2} \ge 16^{-1} |I_j|^{-2}
\]
and
\[
 \frac{ |x-\beta|}{|x_j-\beta| } \ge \min\left\{ \frac{ |x_{j+1}-y'|}{|x_{j+1}-  x_j|} , \frac{ |x_{j-2}-\lambda|}{|x_{j-2}-x_j|} \right\} \ge
 \frac{1}{4} .
\]
Hence,
\[
(-1)^J H(x) =  |H(x)| \ge 4^{-2\mu' - (2r+1)s - (2r+2)p } := c_2 , \quad x\in [y',\lambda] .
\]
This implies that, for  the function
\[
h''   := h'   + \widetilde H  ,  \quad \text{where $\widetilde H(x):= (c_1/c_2)  H(x) |I_j|^m$,}
\]
we have
\[
(-1)^J h''(x) \ge - \norm{h'}{} + (c_1/c_2)  (-1)^J H(x) |I_j|^m \ge 0, \quad x\in [y',\lambda],
\]
and so $h'' \in \Delta^{(0)}(Y_s)$.

It remains to notice that polynomials
\[
\widetilde R_1''  := R_1''   +   \widetilde H  \andd \widetilde R_2''  := R_2''   -   \widetilde H
\]
satisfy \ineq{ii1}-\ineq{ii3}, and the proof is now complete.
\end{proof}

\subsection{Approximation of general splines}

If $V_m=\{v_i\}_{i=1}^m$  is some finite collection of points in $I$, then
we  denote by $\Sigma_{k}(\zn, V_m)$ the subset
of $\Sigma_{k}(\zn)$ consisting of those continuous piecewise polynomials $S$ that do not have any knots ``too close" to the
points in $V_m$.
More precisely, if
  $j_i$, $1\leq i\leq m$, are chosen so that $v_i\in[z_{j_i},z_{j_i-1})$ (or $v_i\in[z_{j_i},z_{j_i-1}]$ if $v_i=1$, \ie  $j_1:=1$ in this case),
then $S$ is in $\Sigma_{k}(\zn, V_m)$ if and only if $S\in \Sigma_{k}(\zn)$ and, for every  $1\leq i\leq m$, the restriction of $S$ to
 $(z_{j_i+1},z_{j_i-2})$  is a polynomial.

With this notation,
the following lemma which is a corollary of a more general result \cite{hky}*{Theorem 4}, can be stated as follows.

  \begin{lemma}[see \cite{hky}*{Theorem 4}]   \label{indc3}
Let $s, m\in\N_0$,  $Y_s \in\Y_s$,
  $\mu \geq 2m+30$,   and let
$S \in \Sigma_{2m+1}(\tn, Y_s)\cap C^{2m-1}$,
 where   $n> c(Y_s)$ is such that there are
at least $4$ knots $x_j$ in each interval $(y_{i+1}, y_{i})$, $0\le i\le s$.

Then, there exists an intertwining pair
of polynomials $\{P_1, P_2\}\subset \Poly_{c(\mu)n}$ for $S$ with respect to
$Y_s$ such that
\[
|P_1(x)-P_2(x)|  \leq  c(m,\mu,s) \sum_{j=1}^{n-1}
\psi_j(x)^\mu
E_{2m}(S, \un) ,
%\left( \frac{|I_j|}{|x-x_j|+|I_j|}\right)^\mu ,
\]
where
$E_n(f, [a,b])  := \inf_{P_n\in\Pn} \|f-P_n\|_{C[a,b]}$.
\end{lemma}

Exactly the same proof also works with an additional condition that polynomials $P_1$ and $P_2$ (Hermite) interpolate $S$ at the points in $Y_s \cup A_p \cup \{\pm 1\}$. Moreover, in view of \lem{newlemma}, the restriction that the order of the spline $S$  is odd is no longer required, and we arrive at the following result.

  \begin{lemma}    \label{apprsplines}
Let $k\in\N$, $p,s, r\in\N_0$,  $Y_s \in\Y_s$, $A_p\in\A_p$,
  $\mu \geq k+30$,   and let
$S \in \Sigma_{k}(\tn, Y_s\cup A_p \cup\{\pm 1\})\cap C^{k-2}$,
 where    $n> c(d(Y_s, A_p))$ is such that there are
at least $4$ knots $x_j$
in each nonempty open interval $(\beta_j,  \beta_{j-1})$, $1\le j \le s+p+1$, where $\{\beta_j\}_{j=1}^{s+p} =  Y_s\cup A_p$, $\beta_0 = 1$ and $\beta_{s+p+1}=-1$  (see \ineq{unionya}).

Then, there exists an intertwining pair
of polynomials $\{P_1, P_2\}\subset \Poly_{c(\mu)n}$ for $S$ with respect to
$Y_s$ such that
\[
 P_1, P_2 \in \I^{(r)}(S, A_p\cup Y_s \cup\{\pm 1\})
 \]
 and
\[
|P_1(x)-P_2(x)|  \leq  c(k,r,\mu,s,p) \sum_{j=1}^{n-1}
\psi_j(x)^\mu  E_{k-1}(S, \un) .
%\left( \frac{|I_j|}{|x-x_j|+|I_j|}\right)^\mu .
\]
\end{lemma}

 \sect{Proofs of the main results}

Theorems~\ref{th:inti}/\ref{thone} and  \ref{thcopositive}/\ref{thpositive}  immediately follow from \lem{apprsplines} and, respectively, Theorems~\ref{splnthm1} and \ref{splcopositive}, using the same sequence of estimates as in \cite{hky}*{Section 3}.

\sect{Appendix: Exact pointwise estimates} \label{point}

It is clear that if a polynomial $P_n$ not only approximates $f$ but also  interpolates it (and perhaps its derivatives) then the rate of approximation \ineq{classdir}
can be improved near all interpolation points.

  The following theorem follows from the results obtained in  \cite{kls21}.

\begin{theorem}[\cite{kls21}*{Theorems 1.5, 1.6 and 1.7}]  \label{maingen}
 Let $k \in\N$, $x_0\in I$,  $r,n\in\N_0$,   $f\in C^r$, and suppose that $m\in\N_0$ is such that $m\le r$. If $P_n\in \Pn$  is such that
 \be \label{classKK}
|f(x)-P_n(x)|\le A\rho_n^r(x)\w_k(f^{(r)},\rho_n(x)),\quad x\in I ,
\ee
 and
  \be \label{2g}
P_n^{(j)}(x_0)=f^{(j)}(x_0), \quad \text{for }\; 0\le j \le m,
\ee
  then, for all
  $1\le \ell \le k$ and $x\in I$, we have
\begin{align}  \label{4nnn}
  | & f (x)  -P_n (x)|  \\ \nonumber
   &
 \le c(k,r) A
\begin{cases}
|x-x_0|^{m+1} \rho_n^{r -m-1}(x)  \omega_\ell(f^{(r)}, \rho_n(x)) , & \text{if }\; m \le r-1 ,\\
|x-x_0|^{r} \omega_\ell(f^{(r)},|x-x_0|^{1/\ell} \rho_n^{1-1/\ell}(x)) , & \text{if }\; m=r .
\end{cases}
\end{align}
Moreover, estimates in \ineq{4nnn} cannot be improved in the sense that none of the powers of $|x-x_0|$   can be increased.
\end{theorem}

The following result immediately follows from  Theorems~\ref{th:inti}, \ref{thone} and \ref{maingen} observing that, for any $(r,m_1,m_2, m_3)\in\intset_{\text{intertwining}}$, we have $m_i < r$, $i=1,2,3$, and so the case $m=r$ in the statement of \thm{maingen} does not hold.
Note also that we are incorporating the onesided case into a  general result for intertwining approximation by allowing $s$ to be $0$, in which case the restriction $P_n\in \I^{(m_2)}(f, Y_0)$ becomes vacuous (recall that the restriction
$P_n\in \I^{(-1)}(f, A_p)$ is also vacuous by definition).

\begin{corollary}[intertwining  approximation with interpolatory constraints]\leavevmode\label{corinti}%
Let $k,p \in\N$, $s\in\N_0$,   $Y_s \in\Y_s$, $A_p\in\A_p$, and $(r,m_1,m_2, m_3)\in\intset_{\text{intertwining}}$, %with $\intset_{\text{intertwining}}$
defined in \ineq{defintertwining}.
Then, for any  $f\in C^r$ and any $n\ge c(d(Y_s, A_p))$, there exists
\[
  P_n \in \Pn\cap \td\cap  \I^{(m_1)}(f, A_p)\cap \I^{(m_2)}(f, Y_s)  \cap  \I^{(m_3)}(f, \{\pm 1\})
\]
 such that
\[
|f(x)-P_n(x)| \leq   c(k,r,s,p) \rho_n^r(x) \w_k(f^{(r)},\rho_n(x)), \qquad x\in I.
\]
Moreover, for all  $1\le \ell \le k$ and  $\lambda \in \Y_s\cup A_p \cup \{\pm 1\}$, we also have
\[
  | f (x)  -P_n (x)|
 \le c(k,r,s,p) |x-\lambda|^{m+1} \rho_n^{r -m-1}(x)  \w_\ell(f^{(r)}, \rho_n(x)) , \quad x\in I,
\]
where
\[
m :=
\begin{cases}
m_1, & \text{if $\lambda \in A_p$,}\\
m_2, & \text{if $\lambda \in Y_s$,}\\
m_3, & \text{if $\lambda =\pm 1$.}
\end{cases}
\]
\end{corollary}

We also have a similar   result for (co)positive approximation with interpolatory constraints which follows from Theorems~\ref{thcopositive}, \ref{thpositive} and \ref{maingen}.
The only difference is that there are a few cases for $r=0$ and $r=1$ when $m_i$'s can be equal to $r$, and so the estimates need to be modified accordingly.

\begin{corollary}[copositive   approximation with interpolatory constraints]\leavevmode\label{corcopositive}%
Let $p \in\N$, $s\in\N_0$,    $Y_s \in\Y_s$, $A_p\in\A_p$, and $(r,k,m_1,m_2, m_3)\in\copset_{\text{copositive}}$, defined in \ineq{defcopositive}.
Then, for any  $f\in C^r\cap  \Delta^{(0)}(Y_s)$ and any $n\ge c(d(Y_s, A_p))$, there exists
\[
  P_n \in \Pn\cap \Delta^{(0)}(Y_s) \cap  \I^{(m_1)}(f, A_p)\cap \I^{(m_2)}(f, Y_s)  \cap  \I^{(m_3)}(f, \{\pm 1\})
\]
 such that
\[
|f(x)-P_n(x)| \leq   c(k,r,s,p) \rho_n^r(x) \w_k(f^{(r)},\rho_n(x)), \qquad x\in I.
\]
 Moreover,
if  $\lambda \in  Y_s\cup A_p \cup \{\pm 1\}$ and
\[
m :=
\begin{cases}
m_1, & \text{if $\lambda \in A_p$,}\\
m_2, & \text{if $\lambda \in Y_s$,}\\
m_3, & \text{if $\lambda =\pm 1$.}
\end{cases}
\]
is such that $m< r$, then for all
 for all  $1\le \ell \le k$, we also have
\[
  | f (x)  -P_n (x)|
 \le c(k,r,s,p) |x-\lambda|^{m+1} \rho_n^{r -m-1}(x)  \w_\ell(f^{(r)}, \rho_n(x)) , \quad x\in I.
\]
Additionally, if $r=1$ and $\lambda \in Y_s$ then, for $\ell=1$ or $\ell=2$,
\[
 |f(x) -P_n(x)| \le c(s,p)  |x-\lambda| \w_\ell(f' ,|x-\lambda|^{1/\ell} \rho_n^{1-1/\ell}(x)), \quad x\in I,
\]
and, if $r=0$ and  $\lambda \in  Y_s\cup A_p\cup \{\pm 1\}$  then, for $\ell=1$ or $\ell=2$,
\[
 |f(x) -P_n(x)| \le c(s,p)  \w_\ell(f ,|x-\lambda|^{1/\ell} \rho_n^{1-1/\ell}(x)), \quad x\in I.
\]
\end{corollary}

We note that, for $r=0$, $A_p = \{-1,1 \}$ and $k=\ell=2$, \cor{corcopositive} was proved in \cite{dz22}.

\end{document}